%% file: main.tex
\documentclass[12pt,a4paper]{amsart}
\usepackage[foot]{amsaddr}

\makeatletter
\input{commands.tex}
\input{symbols.tex}
\makeatother

\bibliography{bibliography}

\begin{document}
	\title[Breather solutions to nonlinear Maxwell equations]{Breather solutions to nonlinear Maxwell equations with retarded material laws}

	\author{Sebastian Ohrem$^1$}
	\email{sebastian.ohrem@kit.edu}
	\address{$^1$Institute for Analysis, Karlsruhe Institute of Technology (KIT), D-76128 Karlsruhe, Germany}
	
	\date{\today}
	\subjclass[2020]{Primary: 35Q61, 49J10; Secondary: 35C07, 78A50}
	\keywords{Maxwell equations, retarded nonlinear material law, polychromatic breather solution, variational method}

	\begin{abstract}
		\input{abstract.tex}
	\end{abstract}

	\maketitle

	\input{introduction.tex}

	\input{variational_setup.tex}

	\input{Lp_embeddings.tex}

	\input{dual_method.tex}

	\input{regularity.tex}
	
	\appendix
	\input{appendix.tex}
	Funded by the Deutsche Forschungsgemeinschaft (DFG, German Research Foundation) – Project-ID 258734477 – SFB 1173

	\printbibliography
\end{document}

%% file: commands.tex
\usepackage[english]{babel}
\usepackage[utf8]{inputenc}
\usepackage{csquotes}
\usepackage[T1]{fontenc}
\usepackage{latexsym}
\usepackage[leqno]{amsmath}
\usepackage{amssymb,amsthm,amsfonts}
\usepackage{mathtools}
\usepackage{mathrsfs}
\usepackage{dsfont}
\usepackage{stmaryrd}
\usepackage{cancel}
\usepackage{xcolor}

\usepackage{todonotes}

\usepackage{geometry}
\usepackage{nicefrac}

\usepackage{enumerate}
\usepackage[shortlabels]{enumitem}
\setlist[enumerate]{label=(\alph*),font=\normalshape}
\setlist[itemize]{font=\normalshape}
\let\originalitem\item
\renewcommand{\item}[1][]{%
	\if\relax\detokenize{#1}\relax%
		\originalitem%
	\else%
		\originalitem[#1]%
		\phantomsection
		\def\@currentlabel{#1}
	\fi%
}

\usepackage{hyperref}

\usepackage[style=numeric-comp]{biblatex}

\usepackage[noabbrev,capitalize]{cleveref}

\let\originalleft\left
\let\originalright\right
\renewcommand{\left}{\mathopen{}\mathclose\bgroup\originalleft}
\renewcommand{\right}{\aftergroup\egroup\originalright}

\allowdisplaybreaks

\theoremstyle{plain}
\newtheorem{theorem}{Theorem}[section]
\newtheorem{definition}[theorem]{Definition}
\newtheorem{proposition}[theorem]{Proposition} 
\newtheorem{lemma}[theorem]{Lemma} 

\theoremstyle{definition}
\newtheorem{remark}[theorem]{Remark}

\AddToHook{env/definition/begin}{\crefalias{theorem}{definition}}
\AddToHook{env/proposition/begin}{\crefalias{theorem}{proposition}}
\AddToHook{env/lemma/begin}{\crefalias{theorem}{lemma}}
\AddToHook{env/corollary/begin}{\crefalias{theorem}{corollary}}
\AddToHook{env/notation/begin}{\crefalias{theorem}{notation}}
\AddToHook{env/remark/begin}{\crefalias{theorem}{remark}}

\newcommand{\C}{\mathbb{C}} 
\newcommand{\R}{\mathbb{R}} 
\newcommand{\Z}{\mathbb{Z}} 
\newcommand{\Zodd}{\mathbb{Z}_\mathrm{odd}}

\newcommand{\N}{\mathbb{N}} 
\newcommand{\Nodd}{\mathbb{N}_\mathrm{odd}}

\newcommand{\T}{\mathbb{T}}
\newcommand{\F}{\mathcal{F}}

\DeclareMathAlphabet{\othermathbb}{U}{bbold}{m}{n}
\newcommand{\bbone}{\othermathbb{1}}

\newcommand{\der}[2][]{\@ifnextchar\der{\,#1\mathrm{d}#2\!}{\,#1\mathrm{d}#2}}
\newcommand{\Der}{\mathrm{d}}
\DeclareMathOperator{\dist}{dist}
\DeclareMathOperator{\supp}{supp}

\DeclareMathOperator*{\essinf}{ess~inf}
\DeclareMathOperator{\sign}{sign}
\renewcommand{\Re}{\operatorname{Re}}

\newcommand{\landauo}{\hbox{o}}
\newcommand{\ee}{\mathrm{e}}
\newcommand{\ii}{\mathrm{i}}
\let\eps\varepsilon
\let\bd\partial
\let\wto\rightharpoonup
\let\embeds\hookrightarrow

\newcommand{\set}[1]{{\left\{ #1 \right\}}}
\newcommand{\Set}[1]{\bigl\{ #1 \bigr\}}
\newcommand{\abs}[1]{\left\lvert #1 \right\rvert}
\newcommand{\Abs}[1]{\bigl\lvert #1 \bigr\rvert}
\newcommand{\norm}[1]{\left\lVert #1 \right\rVert}
\newcommand{\floor}[1]{\left\lfloor #1 \right\rfloor}
\newcommand{\ceil}[1]{\left\lceil #1 \right\rceil}

\newcommand{\impvar}{\,\cdot\,}
\newcommand{\ip}[2]{\left\langle #1 , #2 \right\rangle}

%% file: symbols.tex
\newcommand{\loc}{\mathrm{loc}}
\newcommand{\per}{\mathrm{per}}
\newcommand{\ess}{\mathrm{ess}}
\newcommand{\bfD}{\mathbf{D}}
\newcommand{\bfE}{\mathbf{E}}
\newcommand{\bfF}{\mathbf{F}}
\newcommand{\bfB}{\mathbf{B}}
\newcommand{\bfH}{\mathbf{H}}
\newcommand{\bfP}{\mathbf{P}}

\newcommand{\nlsuffix}[1]{\if#11{.1}\else{.2}\fi}
\newcommand{\calG}{\mathcal{G}}
\newcommand{\calN}{\mathcal{N}}
\newcommand{\calE}{\mathcal{E}}
\DeclareMathOperator*{\periodize}{Per}

\newcommand{\wop}{\mathcal{L}}
\newcommand{\iwop}[1][h]{\mathcal{L}^{-1}_{#1}}
\newcommand{\wdom}{\mathcal{H}}
\newcommand{\wdense}{\mathcal{D}}
\newcommand{\rfreq}{\mathfrak{R}}
\newcommand{\rproj}{P_\mathfrak{R}}

\newcommand{\el}{\mathfrak{c}}
\newcommand{\elmp}{\mathfrak{c}_\mathrm{mp}}
\newcommand{\elgs}{\mathfrak{c}_\mathrm{gs}}

%% file: abstract.tex
We consider Maxwell's equations for Kerr-type optical materials, which are magnetically inactive and have a nonlinear response to electric fields. This response consists of a linear plus a cubic term, which are both inhomogeneous with bounded coefficients. The cubic term is temporally retarded while the linear term has instantaneous and retarded contributions.
For slab waveguides we show existence of breathers, which are time-periodic, real-valued solutions that are localized in the direction perpendicular to the waveguide, and moreover they are traveling along one direction of the waveguide.
We find these breathers using a variational method which relies on the assumption that an effective operator related to the linear part of Maxwell's equations has a spectral gap about $0$. We also give examples of material coefficients, including nonperiodic materials, where such a spectral gap is present.

%% file: introduction.tex

\section{Introduction and main results}

We consider Maxwell's equations 
\begin{align}\label{eq:maxwell}
	\begin{aligned}
		\nabla \cdot \bfD &= 0, \qquad& \nabla \times \bfE &= - \bfB_t,
		\\ \nabla \cdot \bfB &= 0, \qquad& \nabla \times \bfH &= \bfD_t	
	\end{aligned}
\end{align}
in $\R^3$ without changes and currents.
For the underlying material, we assume the constitutive relations
\begin{align}\label{eq:material:generic}
	\bfB = \mu_0 \bfH, \qquad \bfD = \epsilon_0 \bfE + \epsilon_0 \bfP(\bfE)
\end{align}
where $\mu_0, \epsilon_0 > 0$ denote vacuum permeability and permittivity. So the material is magnetically inactive and electrically active, with an electric displacement field $\bfD$ that depends nonlinearly on the electric field $\bfE$ through the polarization $\bfP(\bfE)$. We consider Kerr optical materials modelled by $\bfP(\bfE)$ consisting of a linear plus a cubic term of $\bfE$: the quadratic term is zero for silica glasses, and higher-order terms are omitted (cf. \cite{agrawal}). More precisely, we assume that the polarization is given either by
\refstepcounter{equation}
\begin{align}\label{eq:polarization:1}\tag{\theequation\nlsuffix1}
	\bfP(\bfE)(x, y, z, t) &= \int_0^\infty g(x, \tau) \bfE(x, y, z, t - \tau) \der \tau
	+ h(x) \int_0^\infty \nu(\tau) \bfE(x, y, z, t - \tau)^3 \der \tau
\intertext{where we abbreviate $\bfE^3 \coloneqq \abs{\bfE}^2 \bfE$, or by}
\label{eq:polarization:2}\tag{\theequation\nlsuffix2}
	\bfP(\bfE)(x, y, z, t) &= \int_0^\infty g(x, \tau) \bfE(x, y, z, t - \tau) \der \tau
	+ h(x) \left(\int_0^\infty \nu(\tau) \bfE(x, y, z, t - \tau) \der \tau\right)^3.
\end{align}
Note that the material coefficients $g, h, \nu$ depend only on one spatial direction $x$, which we call a slab material. Taking the curl of Faraday's law $\nabla \times \bfE = - \bfB_t$, one obtains from \eqref{eq:maxwell}, \eqref{eq:material:generic} the second order Maxwell's equation
\begin{align}\label{eq:second_order_maxwell}
	\nabla \times \nabla \times \bfE + \epsilon_0 \mu_0 \partial_t^2 \left(\bfE + \bfP(\bfE)\right) = 0.
\end{align} 

There are many results on breathers for nonlinear wave-type equations like \eqref{eq:second_order_maxwell} in the literature.
Monochromatic breathers, which are given by $\bfE(x, y, z, t) = \Re[\calE(x,y,z) \ee^{\ii \omega t}]$ with frequency $\omega > 0$ and profile $\calE \colon \R^3 \to \C^3$, reduce \eqref{eq:second_order_maxwell} to the elliptic problem 
\begin{align}\label{eq:time_harmonic_maxwell}
	\nabla \times \nabla \times \calE + (\chi_1(x,y,z) + \chi_3(x,y,z) \abs{\calE}^2) \calE = 0
\end{align}
with appropriate functions $\chi_1, \chi_3$ derived from \eqref{eq:second_order_maxwell} by neglecting higher order harmonics, i.e., terms proportional to $\ee^{\pm 3 \ii \omega t}$ in $\bfP(\bfE)$. 
Alternatively, higher order harmonics vanish if the nonlinear part of the polarization is given by the time-average
\begin{align*}
	\bfP_\mathrm{NL}(\bfE)(x,y,z,t) = h(x,y,z) \int_0^T \abs{\bfE(x,y,z,\tau)}^2 \der \tau \bfE(x,y,z,t),
\end{align*}
where $T \coloneqq \frac{2 \pi}{\omega}$. Saturated nonlinearities
\begin{align*}
	\nabla \times \nabla \times \calE + \chi(x,y,z, \abs{\calE}^2) \calE = 0,
\end{align*}
which are asymptotically linear as $\abs{\calE} \to \infty$, are also of interest. These were considered in a series of papers \cite{stuart90,stuart04,stuart_zhou96,stuart_zhou03,stuart_zhou10,stuart_zhou01,stuart_zhou05} by Stuart and Zhou. The authors considered transverse electric (TE) or transverse magnetic (TM) polarized waves in cylindrically symmetric waveguides, which reduce \eqref{eq:time_harmonic_maxwell} to a one-dimensional scalar equation that can, e.g., be treated variationally. 
More general nonlinearities $\chi$, including also power nonlinearities were investigated in \cite{bartsch_dohnal_plum_reichel,benci_fortunato,azzollini_et_al,hirsch_reichel17} for cylindrically or spherically symmetric solutions. 
The restriction to symmetric solutions can be overcome using a Helmholtz decomposition to deal with the kernel of the curl-curl-operator. This was investigated by Mederski et al. in a series of papers \cite{mederski_reichel,mederski15,mederski_schino22,mederski_schino24}. 
Mandel combined the dual variational method with Helmholtz decomposition in \cite{mandel_dual}, and considered spatially nonlocal nonlinearities in \cite{mandel_dual_nonlocal}. 
In \cite{dohnal_romani_bifurcation,dohnal_romani_bifurcation_corr} Dohnal and Romani obtained breathers by bifurcation from a simple eigenvalue of the linear problem. We refer to the survey paper \cite{bartsch_mederski_survey} for further results on monochromatic Maxwell equations.

We move to the topic of polychromatic breathers, which have multiple (usually infinitely many) supported frequencies and are given by
\begin{align*}
	\bfE(x, y, z, t) = \sum_{k \in \Z} \calE_k(x, y, z) \ee^{\ii k \omega t}.
\end{align*}

Let us first discuss the instantaneous nonlinearity $\bfP_\mathrm{NL}(E) = h(x,y,z) \bfE^3$. Here, the authors of \cite{dohnal_schnaubelt_tietz} considered breathers at an interface between two dielectrics and showed that these can be approximated on large but finite time scales by solutions of an amplitude equation, the nonlinear Schrödinger equation. 
Existence of true time-periodic solutions was shown in \cite{kohler_reichel,bruell_idzik_reichel} by Reichel et al. for materials where the linear or nonlinear part of the polarization consists of Dirac measures in space and using variational methods or bifurcation theory, respectively. 
In \cite{ohrem_reichel_quasilinearmaxwell} we considered bounded material coefficients with spatially localized nonlinear interaction, and obtained breathers variationally for instantaneous as well as some time-averaged polarizations. 

Regarding the retarded polarizations \eqref{eq:polarization:1} and \eqref{eq:polarization:2}, 
breathers were found in our recent paper \cite{ohrem_reichel_elliptic} using variational methods. 
The main difference between \cite{ohrem_reichel_elliptic} and this paper is a central property of the effective linear operator (see \eqref{eq:def:wop}). We consider hyperbolic operators with a spectral gap about $0$ while \cite{ohrem_reichel_elliptic} dealt with the elliptic case.

We solve the second-order Maxwell problem \eqref{eq:second_order_maxwell} using the polychromatic ansatz
\begin{align}\label{eq:maxwell_ansatz}
	\bfE(x, y, z, t) = w(x, t - \frac{1}{c} z) \cdot \begin{pmatrix}
		0 \\1 \\ 0
	\end{pmatrix}.
\end{align}
of a TE-polarized wave traveling with speed $c$ in $z$-direction.
The ansatz is divergence-free, so the curl-curl operator simplifies to $\nabla \times \nabla \times \bfE = -\Delta \bfE$. Normalizing the speed of light in vacuum to $c_0 \coloneqq (\epsilon_0 \mu_0)^{-\frac12} = 1$ and inserting into \eqref{eq:second_order_maxwell}, we obtain 
\begin{align}\label{eq:scalar_wave:generic}
	- \partial_x^2 w - \frac{1}{c^2} \partial_t^2 w + \partial_t^2 (w + P(w)) = 0
\end{align}
where the scalar polarization $P(w)$ is given by
\refstepcounter{equation}
\begin{align}\label{eq:scalar_polarization:1}\tag{\theequation\nlsuffix1}
	P(w)(x, t) &= \int_0^\infty g(x, \tau) w(x, t - \tau) \der \tau
	+ h(x) \int_0^\infty \nu(\tau) w(x, t - \tau)^3 \der \tau
\intertext{or}
\label{eq:scalar_polarization:2}\tag{\theequation\nlsuffix2}
	P(w)(x, t) &= \int_0^\infty g(x, \tau) w(x, t - \tau) \der \tau
	+ h(x) \left(\int_0^\infty \nu(\tau) w(x, t - \tau) \der \tau\right)^3,
\end{align}
corresponding to \eqref{eq:polarization:1} and \eqref{eq:polarization:2}, respectively. To include instantaneous linear material responses, we assume a decomposition
\begin{align}\label{eq:g_decomposition}
	g(x, \tau) = g_0(x) \delta_0(\tau) + g_1(x, \tau)
\end{align}
with $\delta_0$ being the Dirac measure at $0$ and $g_0, g_1$ bounded. 
On the other hand, we assume that the nonlinear material response has no instantaneous contribution.

We show existence of breather solutions which solve Maxwell's equations pointwise and are infinitely differentiable in time.
The precise definition is given next. 

\begin{definition}\label{def:smooth_solution}
	We call $\bfE, \bfD, \bfB, \bfH \colon \R^3 \times \R \to \R^3$ \emph{breather solutions} to Maxwell's equations with polarization \eqref{eq:polarization:1} \textup{[}or \eqref{eq:polarization:2}\textup{]} with time-period $T > 0$ traveling with speed $c$ in $z$-direction if each field $\bfF \in \set{\bfE, \bfD, \bfB, \bfH}$ satisfies
	\begin{align*}
		\bfF(x, y, z, t + T)
		= \bfF(x, y, z, t) 
		= \bfF(x, y, z + c \tau, t + \tau)
	\end{align*} 
	and if for all domains of the form $\Omega = \R \times [y_1, y_2] \times [z_1, z_2] \times [t_1, t_2]$ it has the regularity
	\begin{align*}
		\partial_t^n \partial_x^m \bfF \in L^2(\Omega; \R^3) \cap L^\infty(\Omega; \R^3)
	\end{align*}
	for $n \in \N_0$ and $m \in \set{0, \dots, \overline{m}(\bfF)}$ where $\overline{m}(\bfE) = 2, \overline{m}(\bfD) = 0,\overline{m}(\bfB) = \overline{m}(\bfH) = 1$. Moreover, we require \eqref{eq:maxwell}, \eqref{eq:material:generic}, and \eqref{eq:polarization:1} \textup{[}or \eqref{eq:polarization:2}\textup{]} to hold pointwise almost everywhere.
\end{definition}

Next we give two examples of material parameters $g_0, g_1, h, \nu$ for which we can show existence of breather solutions. 
In the first example, we consider a spatially periodic linear material response, i.e., $g_0$ and $g_1$ are periodic in $x$.

\begin{theorem}\label{thm:example:periodic}
	Let $c \in (0, \infty), \theta \in (0, 1) \setminus \Set{\frac12}, T, X > 0$. Further let $g_1^\per, h^\per, h^\loc \in L^\infty(\R; \R)$ such that $g_1^\per, h^\per$ are $X$-periodic, $h^\per$ is positive almost everywhere, $h^\loc \geq 0$ and $h^\loc(x) \to 0$ as $\abs{x} \to \infty$. We set $\omega \coloneqq \frac{2 \pi}{T}$ and define potentials $g_0, g_1, h, \nu$ by
	\begin{align*}
		g_0(x) &\coloneqq g_0(x; \theta, X) \coloneqq \frac{1}{c^2} - 1 + \begin{cases}
			\frac{T^2}{16 \theta^2 X^2}, & x \in (0, \theta X) + X \Z, 
			\\ \frac{T^2}{16 (1 - \theta)^2 X^2}, & x \in (\theta X, X) + X \Z,
		\end{cases}
		\\ g_1(x, t) &\coloneqq g_1^\per(x) \cos(\omega t) \abs{\cos(\omega t)} \bbone_{[0, T]}(t),
		\\ h(x) &\coloneqq h^\per(x) + h^\loc(x),
		\\ \nu(t) &\coloneqq \dist(t, T \Z) \bbone_{[0, T]}(t)
	\end{align*}
	for $x, t \in \R$.
	Then for polarization \eqref{eq:polarization:1} as well as polarization \eqref{eq:polarization:2}, there exist infinitely many distinct breather solutions with period $T$ and speed $c$ in the sense of \cref{def:smooth_solution}.
\end{theorem}

In the second example, we consider a linear material response that is spatially periodic on each halfspace, and a localized nonlinear response.

\begin{theorem}\label{thm:example:localized}
	Let $c \in (0, \infty), \theta^-, \theta^+ \in (0, \frac12)$ and $T, X^-, X^+ > 0$. Further let $g_1^\per, h^\loc \in L^\infty(\R; \R)$ such that $g_1^\per$ is $X^-$-periodic on $(-\infty, 0)$ and $X^+$-periodic on $(0, \infty)$, $h^\loc$ is almost everywhere positive and $h^\loc(x) \to 0$ as $\abs{x} \to \infty$. Set $\omega \coloneqq \frac{2 \pi}{T}$ and define $g_0, g_1, h, \nu$ by
	\begin{align*}
		g_0(x) &= \begin{cases}
			g_0(x; \theta^-, X^-), & x < 0, 
			\\ g_0(x; \theta^+, X^+), & x > 0, 
		\end{cases}
		\\ g_1(x, t) &= g_1^\per(x) \cos(\omega t) \abs{\cos(\omega t)} \bbone_{[0, T]}(t),
		\\ h(x) &= h^\loc(x),
		\\ \nu(t) &= \dist(t, T \Z) \bbone_{[0, T]}(t)
	\end{align*}
	for $x, t \in \R$, where $g_0(x; \theta, X)$ is given by \cref{thm:example:periodic}.
	Then for \eqref{eq:polarization:1} as well as \eqref{eq:polarization:2} there exist infinitely many distinct breather solutions with period $T$ and speed $c$ in the sense of \cref{def:smooth_solution}.
\end{theorem}

Let us prepare the main theorem. 
Recall \eqref{eq:scalar_wave:generic}, which using \eqref{eq:scalar_polarization:1}, \eqref{eq:scalar_polarization:2}, and \eqref{eq:g_decomposition} we rewrite as 
\begin{align}\label{eq:scalar_wave}
	\left[-\partial_x^2 + (1 - \tfrac{1}{c^2} + g_0(x)) \partial_t^2\right] w + \partial_t^2 (g_1 \ast w + P_\mathrm{NL}(w)) = 0
\end{align}
where $\ast$ denotes convolution in time and $P_\mathrm{NL}$ is the nonlinear part of the scalar polarization. We consider velocities $c$ that are so large that the potential
\begin{align}\label{eq:def:V}
	V(x) \coloneqq 1 - \frac{1}{c^2} + g_0(x)
\end{align}
is positive, hence the linear operator $-\partial_x^2 + V(x) \partial_t^2$ is hyperbolic. 
It is convenient to divide this operator by $V(x)$ and instead consider $L + \partial_t^2$ with weighted Sturm-Liouville operator
\begin{align}\label{eq:def:L}
	L \coloneqq - \frac{1}{V(x)} \partial_x^2.
\end{align}
In this form, the spectrum of $L + \partial_t^2$ restricted to time-periodic functions is easier to compute since $L$ acts only on $x$, and therefore we simply have $\sigma(L + \partial_t^2) = \overline{\sigma(L) + \sigma(\partial_t^2)}$. 

Let us fix the period $T > 0$ of the breather and consider as time-domain the set $\T \coloneqq \R/_{T\Z}$. On it, $\partial_t^2$ has discrete spectrum $\sigma(\partial_t^2) = \set{-\omega^2 k^2 \colon k \in \Z}$ with $\omega \coloneqq \frac{2 \pi}{T}$ denoting the base frequency. Our analysis strongly uses the assumption that $L$ has a spectral gap about $\omega^2 k^2$ for each $k \in \Zodd$. The restriction to odd frequencies means we consider functions that are $\frac{T}{2}$-antiperiodic in time. It is helpful since for $k = 0$ we always have $0 \in \sigma(L)$, $\frac{T}{2}$-antisymmetry is compatible with \eqref{eq:scalar_wave}, it still yields a variational problem, and it allows us to invert the operator $L + \partial_t^2$ after restricting to $\frac{T}{2}$-antiperiodic functions in time.
The operator $L$ having countably many specific spectral gaps requires a careful choice of the potential $g_0$. 
In \cref{thm:example:periodic,thm:example:localized} this is satisfied, and moreover the size of the spectral gaps grows linearly in $\abs{k}$. 
\begin{remark}
	In \cite[Appendix~C]{henninger_ohrem_reichel} you can find further examples of potentials $g_0$ for which the results of \cref{thm:example:periodic,thm:example:localized} remain valid.
	As an example, in \cref{thm:example:periodic} one can consider
	\begin{align*}
		g_0(x) &= \frac{1}{c^2} - 1 + \begin{cases}
			\frac{m^2 T^2}{16 \theta^2 X^2}, & x \in (0, \theta X) + X \Z, 
			\\ \frac{n^2 T^2}{16 (1 - \theta)^2 X^2}, & x \in (\theta X, X) + X \Z,
		\end{cases}
	\end{align*}
	for $m, n \in \Nodd$ and $\theta \in (0, 1)$, where instead of $\theta \neq \frac12$ we require $g_0 \neq \mathrm{const}$. Similar results hold for a periodic arrangement of three or more step potentials. 
\end{remark}

Lastly, let us fix some notation for the torus $\T$. It is equipped with the Haar measure $\Der t$ normalized such that $\int_\T 1 \der t = 1$. 
We denote the standard orthonormal basis on $\T$ by $e_k(t) \coloneqq \ee^{\ii k \omega t}$. Accordingly, the Fourier coefficients of a function $\varphi \colon \T \to \C$ are given by $\hat \varphi_k = \F_k[\varphi] = \int_\T \varphi \overline{e_k} \der t$, and the inverse is $\varphi(t) = \F^{-1}[\hat \varphi_k](t) = \sum_{k \in \Z} \hat \varphi_k e_k(t)$. 
If $\varphi$ depends on space and time, $\hat \varphi$ will always denote its temporal Fourier transform.
Moreover, given a function $f$ on $\R$, we define its periodization by $\periodize[f](t) = T \sum_{k \in \Z} f(t + k T)$ for $t \in \T$. Note that $f \ast_\R \varphi \coloneqq \periodize[f] \ast_\T \varphi$ holds for $T$-periodic functions $\varphi$.

With this, we present our main existence result on breather solutions to Maxwell's equations \eqref{eq:maxwell} and \eqref{eq:material:generic} with polarization \eqref{eq:polarization:1} or \eqref{eq:polarization:2}.

\begin{theorem}\label{thm:main}
	Let $T > 0$ be the period of the breather, $\omega \coloneqq \frac{2 \pi}{T}$ be its frequency and $c \in (0, \infty)$ be its speed.
	Assume that for constants $\alpha > 1$, $\frac12 \leq \beta < 2, \gamma \leq 1$, $0 < d < \delta$ we have:
	\begin{enumerate}[({A}1)]
		\item\label{ass:h}\label{ass:first} $h \in L^\infty(\R; (0, \infty))$.

		\item\label{ass:calN}
		$\nu \in L^1(\R; \R)$ and its periodization $\calN \coloneqq \periodize[\nu]$ is even. 
		Denoting its Fourier support restricted to odd frequencies by $\rfreq \coloneqq \set{k \in \Zodd \colon \hat \calN_k \neq 0}$, we have $\rfreq \neq \emptyset$ and $\Abs{\hat \calN_k} \lesssim \abs{k}^{-\alpha}$ for all $k \in \rfreq$. 

		\medskip
		\item\label{ass:g0} $g_0 \in L^\infty(\R; \R)$ satisfies $\essinf g_0 > \tfrac{1}{c^2} - 1$ and is locally of bounded variation.
		
		\item\label{ass:spec}
		For the spectrum of the operator $L \colon H^2(\R) \to L^2(\R)$ defined in \eqref{eq:def:L} we have:
		\begin{itemize}
			\item $(\omega^2 k^2 - \delta \abs{k}^\gamma, \omega^2 k^2 + \delta \abs{k}^\gamma) \subseteq \rho(L)$ holds for all $k \in \rfreq$.
			
			\item The point spectrum $\sigma_p(L)$ satisfies $\sum_{\lambda \in \sigma_p(L)} \lambda^{-\beta-\eps} < \infty$ for all $\eps > 0$.
		\end{itemize}
		
		\item\label{ass:L4_embed} $\alpha + \gamma - 2 > \beta$.

		\item\label{ass:g1}\label{ass:last_nogeom_u}
		$g_1 \in L^\infty_x(\R; L^1_t(\R; \R))$ and its periodization $\calG(x) \coloneqq \periodize[g_1(x; \impvar)]$ is even in $t$ and satisfies $\Abs{\hat \calG_k(x)} \leq \frac{d}{\omega^2 \abs{k}^{2 - \gamma}} V(x)$ for all $x \in \R$, $k \in \rfreq$.

		\item\label{ass:add_for_polarization_2}\label{ass:last_nogeom} If the polarization is \eqref{eq:polarization:2}, assume that $\Abs{\hat \calN_k} \gtrsim \abs{k}^{-s}$ holds for all $k \in \rfreq$ and for some $s \in \R$. In addition, similar to \ref{ass:spec} and \ref{ass:g1} we then require that there exist constants $\tilde \gamma \leq 1$, $0 < \tilde d < \tilde \delta$ with
		\begin{align*}
			(\omega^2 k^2 - \tilde \delta \abs{k}^{\tilde \gamma}, \omega^2 k^2 + \tilde \delta \abs{k}^{\tilde \gamma}) \subseteq \rho(L),
			\qquad
			\Abs{\hat \calG_k(x)} \leq \frac{\tilde d}{\omega^2 \abs{k}^{2 - \tilde \gamma}} V(x)
		\end{align*}
		for all $k \in \Zodd \setminus \rfreq$.
		(note that $\alpha + \tilde \gamma - 2 > \beta$ is not required)
	\end{enumerate}
	Further let one of the two following assumptions on the spatial geometry of $g_0, g_1, h$ hold.
	\begin{enumerate}[({A8}a)]
		\item\label{ass:localized_nonlin} $h(x) \to 0$ as $\abs{x} \to \infty$. In addition there exist $R^\pm \in \R$, $X^\pm > 0$ such that $g_0$ is periodic on $[R^+, \infty)$ with period $X^+ > 0$, and also on $(-\infty, R^-]$ with period $X^-$.

		\item\label{ass:perturbed_nonlin}  $h = h^\loc + h^\per$ where $h^\loc(x) \geq 0$ satisfies $h^\loc(x) \to 0$ as $\abs{x} \to \infty$, and $h^\per, g_0, g_1$ are periodic in $x$ with common period.
	\end{enumerate}
	Then there exists a nonzero breather solution $\bfE, \bfB, \bfD, \bfH$ to Maxwell's equations with period $T$ and speed $c$ in the sense of \cref{def:smooth_solution}.

	If moreover the set $\rfreq$ is infinite, there exists infinitely many distinct breather solutions.
\end{theorem}

\begin{remark}\label{rem:on_main_thm}
	Let us comment on \cref{thm:main} and its assumptions.
	\begin{itemize}
		\item The constitutive relations \eqref{eq:polarization:1} and \eqref{eq:polarization:2} are translation invariant in time. That is, if $\bfE, \bfD, \bfB, \bfH$ solve the Maxwell system \eqref{eq:maxwell} and \eqref{eq:material:generic}, then for $\tau \in \R$ the shifted functions $\bfE(\impvar, \impvar - \tau), \bfD(\impvar, \impvar - \tau), \bfB(\impvar, \impvar - \tau), \bfH(\impvar, \impvar - \tau)$ also are solutions. In \cref{thm:main}, for $\# \rfreq = \infty$ we state existence of infinitely many \emph{distinct} solution. By this, we mean infinitely many solutions that are not shifts of one another.

		\item For \ref{ass:h}, it is also possible to treat a negative potential $h$ in front of the nonlinearity. More precisely, \cref{thm:main} remains valid when $h$ is replaced by $-h$, and we discuss changes to the proof in \cref{rem:negative_h}.
		However, with our method it is not possible to treat nonlinear potentials $h$ that change sign, or those that vanish on a set of nonzero measure.
	
		\item The evenness assumption in \ref{ass:calN} and \ref{ass:g1} is needed for the variational structure, since the linear part of the variational problem, cf. \eqref{eq:variational_wave}, is nonsymmetric without this assumption.
		Evenness is equivalent to time reversal symmetry of the time-periodic Maxwell equations.

		The growth assumptions on the Fourier coefficients $\hat \calN_k$ in \ref{ass:calN} together with \ref{ass:spec} and \ref{ass:L4_embed} ensure that the variational problem can be treated using semilinear methods. For example, we show that the nonlinear terms are well-defined and finite on the form domain of the linear operator. 

		\item Assumption~\ref{ass:add_for_polarization_2} lets us control the linear operator also along frequencies $k \in \Zodd \setminus \rfreq$. It is not needed for \eqref{eq:scalar_polarization:1} since there the nonlinearity contains a convolution with $\nu$, which projects onto the frequencies $\rfreq$. The lower bound $\abs{k}^{-s}$ on the Fourier coefficients $\hat \calN_k$ gives us control over the inverse of this convolution operator.

		\item The geometry assumptions \ref{ass:localized_nonlin} and \ref{ass:perturbed_nonlin} are used to overcome noncompactness issues for the variational functional: The decay of $h$ in \ref{ass:localized_nonlin} ensures complete continuity of the nonlinearity, whereas the periodic structure in \ref{ass:perturbed_nonlin} allows us to use concentration-compactness arguments. 

		Under \ref{ass:perturbed_nonlin} the assumption on the point spectrum in \ref{ass:spec} is trivially satisfied since the differential operator $L$ is periodic, and therefore by Floquet-Bloch theory (cf. \cite{eastham}) the spectrum $\sigma(L)$ consists of pure essential spectrum.
	\end{itemize}
\end{remark}

\subsection*{Outline of the paper}
In \cref{sec:variational_formulation} we formally convert the Maxwell problem into an Euler-Lagrange equation and investigate the appearing symmetric linear operator $\wop$. 
\cref{sec:embeddings} deals with embeddings of the form domain $\wdom$ of the operator $\wop$, i.e., the natural domain of the bilinear form associated to $\wop$. We show boundedness as well as local compactness of $\wdom \embeds L^4(\R \times \T)$. This allows us choose $\wdom$ as domain of the Lagrangian functional. 
In \cref{sec:dual_problem} we consider the dual problem to the Euler-Lagrange equation, and solve it using the mountain pass method. 
We also discuss multiplicity of solutions claimed in \cref{thm:main}. 
Next, in \cref{sec:regularity} we discuss regularity of solutions of the Euler-Lagrange equation as well as the Maxwell system, proving \cref{thm:main}. 
Lastly, Appendix~\ref{sec:appendix} contains the proofs of \cref{thm:example:periodic,thm:example:localized} as well as some auxiliary results.
Throughout these sections, we always assume that \ref{ass:first}--\ref{ass:last_nogeom} and one of \ref{ass:localized_nonlin} or \ref{ass:perturbed_nonlin} are fulfilled. 

%% file: variational_setup.tex

\section{Variational formulation}\label{sec:variational_formulation}

We begin by transforming the scalar Maxwell problem \eqref{eq:scalar_wave} into a variational problem for an auxiliary variable $u$. We follow \cite{ohrem_reichel_elliptic} for the formal derivation. 

First we consider the polarization \eqref{eq:scalar_polarization:1}. We denote by $\ast$ the convolution on the time-domain $\T$ and use
\begin{align*}
	V(x) = 1 - \frac{1}{c^2} + g_0(x),
	\qquad
	\calG(x) = \periodize[g_1(x, \impvar)],
	\qquad
	\calN = \periodize[\nu]
\end{align*}
to rephrase \eqref{eq:scalar_wave} as
\begin{align}\label{eq:scalar_wave:1}
	\left( -\partial_x^2 + V(x) \partial_t^2 \right) w + \partial_t^2 \calG(x) \ast w + h(x) \partial_t^2 \calN \ast w^3 = 0.
\end{align}
Observe that $V$ is bounded, strictly positive, and locally of bounded variation due to \ref{ass:g0}. 
The nonlinearity $\partial_t^2 \calN \ast w^3$ is supported only on frequencies $k \in \rfreq$, so it is reasonable to assume that $w$ is also supported only on such frequencies. 
We denote the projection onto frequencies $k \in \rfreq$ by $\rproj$, i.e., $\rproj[\varphi] \coloneqq \F^{-1}[\bbone_{k \in \rfreq} \F_k[\varphi]]$. 
On these frequencies, the operators $-\partial_t^2$, $\calN \ast$ are invertible as Fourier multipliers with nonzero symbol, which allows us to rewrite \eqref{eq:scalar_wave:1} as 
\begin{align}\label{eq:scalar_wave:1:variational}
	\left(-\partial_t^2 \calN \ast\right)^{-1} \left(-\partial_x^2 + V(x) \partial_t^2 + \partial_t^2 \calG(x) \ast\right) w - h(x)\rproj[w^3] = 0.
\end{align}
We abbreviate \eqref{eq:scalar_wave:1:variational} to $\wop w - h \rproj[w^3] = 0$ by introducing the linear operator 
\begin{align}\label{eq:def:wop}
	\wop \coloneqq \left(- \partial_t^2 \calN \ast\right)^{-1} \left(-\partial_x^2 + V(x) \partial_t^2 + \partial_t^2 \calG(x) \ast\right).
\end{align}
Note that $\partial_t^2$, $\calN \ast$, and $-\partial_x^2 + V(x) \partial_t^2 + \partial_t^2 \calG(x) \ast$ mutually commute since they act on time as Fourier multipliers.
Since $\calN, \calG(x)$ are even in time, the convolution operators $\calN \ast, \calG(x) \ast$ are symmetric and hence $\wop$ is symmetric.

Let us now consider polarization \eqref{eq:scalar_polarization:2}, where \eqref{eq:scalar_wave} reads
\begin{align}\label{eq:loc:1}
	\left( -\partial_x^2 + V(x) \partial_t^2 \right) w + \partial_t^2 \calG(x) \ast w + h(x) \partial_t^2 (\calN \ast w)^3 = 0.
\end{align}
We substitute $u \coloneqq \calN \ast w$ in \eqref{eq:loc:1} and apply $\rproj$ to see that $u$ solves 
\begin{align*}
	(-\partial_t^2)^{-1} \left(-\partial_x^2 + V(x) \partial_t^2 + \partial_t^2 \calG(x) \ast\right) (\calN \ast)^{-1} u - h(x) \rproj[u^3] = 0,
\end{align*}
whereas by projecting with $\mathrm{Id} - \rproj$ onto \eqref{eq:loc:1} we obtain
\begin{align*}
	\left(-\partial_x^2 + V(x) \partial_t^2 + \partial_t^2 \calG(x) \ast\right) (\mathrm{Id} - \rproj) w = - h(x) \partial_t^2 (\mathrm{Id} - \rproj)[u^3].
\end{align*}
The first of the two equations is again
\begin{align}\label{eq:variational_wave}
	\wop u - h(x) \rproj[u^3] = 0,
\end{align}
and the second together with $\rproj[w] = (\calN \ast)^{-1} u$ allows us to reconstruct the wave profile $w$ via 
\begin{align}\label{eq:wave_reconstruction:polarization2}
	\begin{aligned}
		w 
		&= \rproj[w] + (\mathrm{Id} - \rproj)[w] 
		\\ &= (\calN \ast)^{-1} u + \left( -\partial_x^2 + V(x) \partial_t^2 + \partial_t^2 \calG(x) \ast \right)^{-1} [- h(x) \partial_t^2 (\mathrm{Id} - \rproj)[u^3]].
	\end{aligned}
\end{align}

In particular, for both choices of the polarization we have to solve the problem \eqref{eq:variational_wave}. Then, we have $w = u$ for polarization \eqref{eq:scalar_polarization:1} whereas $w$ is given by \eqref{eq:wave_reconstruction:polarization2} for polarization \eqref{eq:scalar_polarization:2}.

We now study the form domain $\wdom$ of $\wop$. 
For this, we use a functional calculus for the Sturm-Liouville operator $L$ on a weighted $L^2$-space. We introduce both before defining $\wdom$ in \cref{def:H}.

\begin{definition}
	We define the $V$-weighted space $L^2_V(\R; \C) \coloneqq L^2(\R; \C; V \Der x)$. Uniform boundedness and positivity of $V$ show $L^2_V(\R; \C) = L^2(\R; \C)$ with equivalent norms.
\end{definition}

\begin{theorem}[cf. {\cite[Theorem~3.6]{henninger_ohrem_reichel}}] \label{thm:functional_calc}
	Let $\Psi(x; \lambda) = (\Psi_1(x; \lambda), \Psi_2(x; \lambda))^\top$ be the fundamental system of solutions of $L \varphi = \lambda \varphi$ on $\R$ with initial data $\left .(\Psi, \partial_x \Psi) \right\vert_{x = 0} = \mathrm{Id}$. Then there exists a measure $\mu$ defined on the bounded Borel subsets of $\R$ that maps to positive semidefinite $\R^{2 \times 2}$ matrices such that
	\begin{align*}
		T \colon L^2_V(\R; \C) \to L^2(\mu), \quad
		T[f](\lambda) = \int_\R f(x) \Psi(x; \lambda) \der[V]{x}
	\end{align*}
	is an isometric isomorphism with inverse
	\begin{align*}
		T^{-1}[g](x) = \int_\R g_i(\lambda) \Psi_j(x; \lambda) \der \mu_{ij}(\lambda).
	\end{align*}
	Here we use the Einstein summation convention. 
	A definition of the Hilbert space $L^2(\mu)$ consisting of equivalence classes of $\C^2$-valued measurable functions can be found in \cite[Definition XIII.5.8]{dunford_schwartz}. Its norm is given by $\norm{g}_{L^2(\mu)}^2 = \int_\R g_i \overline{g_j} \der \mu_{ij}(\lambda)$.
	We note that the integral defining $T$ exists for compactly supported $f$, and $T$ is defined by approximation for general $f$. The same holds for $T^{-1}$.
\end{theorem}

Let us point out some basic properties of the transform $T$. A proof is given in Appendix~\ref{sec:appendix}.
\begin{lemma}\label{lem:basic_calc_properties}
	Let $f \in L^2(\R; \C)$. Then the following hold:
	\begin{enumerate}
		\item $f \in H^2(\R; \C)$ if and only if $\lambda T[f](\lambda) \in L^2(\mu)$, and we have
		\begin{align*}
			L f = T^{-1}[\lambda T[f](\lambda)]
		\end{align*}

		\item $f \in H^1(\R; \C)$ if and only if $\sqrt{\lambda} T[f](\lambda) \in L^2(\mu)$, and we have
		\begin{align*}
			\int_\R \abs{f'}^2 \der x = \int_\R \lambda T_i[f](\lambda) \overline{T_j[f](\lambda)} \der \mu_{ij}(\lambda) 
		\end{align*}

		\item Moreover, the support of $\mu$ satisfies
		\begin{align*}
			\supp(\mu) 
			\coloneqq \bigcup_{i,j = 1}^2 \supp(\mu_{ij})
			= \sigma(L)
		\end{align*}
	\end{enumerate}
\end{lemma}

Now we rigorously define the bilinear form $b_\wop$ associated to the operator $\wop$ of \eqref{eq:def:wop} and its domain $\wdom$. In the following, we identify the bilinear form $b_\wop \colon \wdom \times \wdom \to \R$ with the weak formulation of the operator $\wop$ via $\wop[u][\varphi] = b_\wop[u, \varphi]$ for $u, \varphi \in \wdom$, i.e., we consider $\wop \colon \wdom \to \wdom'$ where $\wdom'$ is the dual space.

\begin{definition}\label{def:H}
	We define the form domain $\wdom$ by
	\begin{align*}
		\wdom \coloneqq \set{u \in L^2(\R \times \T; \R) \colon \hat u_k = 0 \text{ for } k \in \Z \setminus \rfreq, \ip{u}{u}_\wdom < \infty}
	\end{align*}
	where
	\begin{align*}
		\ip{u}{v}_\wdom
		\coloneqq \sum_{k \in \rfreq} \int_\R \abs{\frac{\lambda - \omega^2 k^2}{\omega^2 k^2 \hat \calN_k}} T_i[\hat u_k](\lambda) \overline{T_j[\hat v_k](\lambda)} \der \mu_{ij}(\lambda).
	\end{align*}
	
	Next, we define the operators $\wop, \wop_0, \wop_1 \colon \wdom \to \wdom'$ by $\wop \coloneqq \wop_0 + \wop_1$ and
	\begin{align*}
		\wop_0[u][\varphi]
		&\coloneqq \sum_{k \in \rfreq} \int_\R \frac{\lambda - \omega^2 k^2}{\omega^2 k^2 \hat \calN_k} T_i[\hat u_k](\lambda) \overline{T_j[\hat \varphi_k](\lambda)} \der \mu_{ij}(\lambda),
		\\
		\wop_1[u][\varphi]
		&\coloneqq \sum_{k \in \rfreq} \int_\R \frac{\hat \calG_k(x)}{\hat \calN_k}  \hat u_k(x) \overline{\hat \varphi_k(x)} \der{x}
	\end{align*}
	for $u, \varphi \in \wdom$.
	We call a function $u$ \emph{weak solution} to \eqref{eq:variational_wave} if $u \in \wdom$ and
	\begin{align*}
		\wop[u][\varphi] - \int_{\R \times \T} h(x) u^3 \varphi \der{(x,t)} = 0
	\end{align*}
	holds for all $\varphi \in \wdom$.
	We show below in \cref{lem:wop_defined,prop:Lp_embedding} that the above integrals and sums are finite and that embedding $\wdom \embeds L^4(\R \times \T)$ and the maps $\wop_0, \wop_1 \colon \wdom \to \wdom'$ are bounded.
\end{definition}

We continue by investigating the operator $\wop$, its domain $\wdom$, and their properties. The following estimate on the symbol $\abs{\lambda - \omega^2 k^2}$ will be useful. 

\begin{remark}\label{rem:abs_symbol_estimate}
	For $k \in \rfreq$ and $\lambda \in \sigma(L)$ we have
	\begin{align*}
		\abs{\lambda - \omega^2 k^2} \geq \delta \abs{k}^\gamma
		\qquad\text{and}\qquad
		\abs{\lambda - \omega^2 k^2} \geq \frac{\delta \abs{k}^\gamma}{\omega^2 k^2 + \delta \abs{k}^\gamma} \lambda.
	\end{align*}
	The first estimate follows directly from \ref{ass:spec}. 
	For the second estimate, we fix $k$ and consider $\frac{\abs{\lambda - \omega^2 k^2}}{\lambda}$ on $\lambda \in \sigma(L) \subseteq [0, \omega^2 k^2 - \delta \abs{k}^\gamma] \cup [\omega^2 k^2 + \delta \abs{k}^\gamma, \infty) \coloneqq I$. We estimate the quotient from below by its minimum value on $I$, which is attained at $\lambda = \omega^2 k^2 + \delta \abs{k}^\gamma$.
\end{remark}

We have the following density result for $\wdom$.

\begin{lemma}\label{lem:dense}
	The set $\wdense \coloneqq \set{u \in \wdom \cap C_c^\infty(\R \times \T; \R) \colon \hat u_k = 0 \text{ for almost all } k \in \rfreq}$ is dense in $\wdom$.
\end{lemma}
\begin{proof}
	Note for $u \in \wdom$ that
	\begin{align*}
		\sum_{\substack{k \in \rfreq \\ \abs{k} \leq K}} \hat u_k(x) e_k(t) \to u	
		\quad\text{in }\wdom
	\end{align*}
	as $K \to \infty$. Moreover, for fixed $k$, using \cref{rem:abs_symbol_estimate} we have
	\begin{align*}
		\abs{\frac{\lambda - \omega^2 k^2}{\omega^2 k^2 \hat \calN_k}} \eqsim_k \lambda + 1
	\end{align*}
	which combined with \cref{lem:basic_calc_properties} shows that $\hat u_k \in H^1(\R; \C)$ and that the norms $\norm{\hat u_k}_{H^1}$ and $\norm{\hat u_k(x) e_k(t)}_\wdom$ are equivalent. As $C_c^\infty(\R; \C) \subseteq H^1(\R; \C)$ is dense, the result follows by approximating $\hat u_k$ for all $\abs{k} \leq K$.
\end{proof}

Lastly, we show that $\wop$ is invertible and indefinite. Both properties are essential for the dual variational method which we use to solve \eqref{eq:variational_wave}.

\begin{lemma}\label{lem:wop_defined}
	$\wop, \wop_0, \wop_1$ are symmetric, $\wop_0$ is an isometric isomorphism and $\norm{\wop_1} < 1$. By the Neumann series, $\wop$ is an isomorphism.
\end{lemma}
\begin{proof}
	From the definitions of $\wdom$ and $\wop_0$ it follows that $\wop_0$ is an isometric isomorphism, and it is clearly symmetric. By \ref{ass:g1} the potential $\calG(x)$ is even in time, so its Fourier coefficients are real-valued and hence $\wop_1$ is symmetric. It remains to show the bound on $\wop_1$. For this, recall that by \cref{lem:basic_calc_properties} the spectral measure $\mu$ is supported on $\sigma(L)$. Using \cref{rem:abs_symbol_estimate} we have
	\begin{align*}
		\norm{u}_\wdom^2
		\geq \sum_{k \in \rfreq} \int_\R \frac{\delta \abs{k}^\gamma}{\omega^2 k^2 \Abs{\hat \calN_k}} T_i[\hat u_k](\lambda) \overline{T_j[\hat u_k](\lambda)} \der \mu_{ij}(\lambda)
		= \sum_{k \in \rfreq} \frac{\delta}{\omega^2 \abs{k}^{2 - \gamma} \Abs{\hat \calN_k}} \norm{\hat u_k}_{L^2_V}^2.
	\end{align*}
	Next, using assumption \ref{ass:g1}, the Cauchy-Schwarz inequality and the above we find
	\begin{align*}
		\abs{\wop_1[u][\varphi]} 
		&\leq \sum_{k \in \rfreq} \frac{1}{\Abs{\hat \calN_k}} \biggl\lVert\frac{\hat \calG_k}{V}\biggr\rVert_\infty \norm{\hat u_k}_{L^2_V(\R)} \norm{\hat \varphi_k}_{L^2_V}
		\leq \sum_{k \in \rfreq} \frac{d}{\omega^2 \abs{k}^{2 - \gamma} \Abs{\hat \calN_k}} \norm{\hat u_k}_{L^2_V} \norm{\hat \varphi_k}_{L^2_V}
		\\ &\leq \left( \sum_{k \in \rfreq} \frac{d}{\omega^2 \abs{k}^{2 - \gamma} \Abs{\hat \calN_k}} \norm{\hat u_k}_{L^2_V}^2 \right)^{\frac12} \left( \sum_{k \in \rfreq} \frac{d}{\omega^2 \abs{k}^{2 - \gamma} \Abs{\hat \calN_k}}\norm{\hat \varphi_k}_{L^2_V}^2 \right)^{\frac12}
		\leq \frac{d}{\delta} \norm{u}_\wdom \norm{\varphi}_\wdom,
	\end{align*}
	showing $\norm{\wop_1} \leq \frac{d}{\delta} < 1$.
\end{proof}

\begin{lemma}\label{lem:indefinite}
	$\wop$ is an indefinite bilinear form.
\end{lemma}
\begin{proof}
	Clearly, the spectrum $\sigma(L)$ of the Sturm-Liouville operator $L$ contains $0$ and is not bounded from above. Fix $k \in \rfreq$. By \ref{ass:spec} and \cref{lem:basic_calc_properties} we find a function $g \in L^2(\mu) \setminus \set{0}$ that is supported on $[\omega^2 k^2 + \delta \abs{k}^\gamma, N]$ for some $N > 0$. The function
	\begin{align*}
		u(x, t) = T^{-1}[g](x) e_k(t) + T^{-1}[\overline{g}](x) e_{-k}(t).
	\end{align*}
	satisfies
	\begin{align*}
		\wop_0[u][u] = 2 \int_\R \frac{\lambda - \omega^2 k^2}{\omega^2 k^2 \hat \calN_k} g_i(\lambda) \overline{g_j(\lambda)} \der \mu_{ij}(\lambda),
	\end{align*}
	which is nonzero with $\sign(\wop_0[u][u]) = \sign(\hat \calN_k)$. Since $\norm{u}_\wdom^2 = \abs{\wop_0[u][u]}$, \cref{lem:wop_defined} shows that $\wop[u][u]$ is nonzero with $\sign(\wop[u][u]) = \sign(\wop_0[u][u]) = \sign(\hat \calN_k)$.

	If above we instead choose $g$ supported on $[0, \omega^2 k^2 - \delta \abs{k}^\gamma]$, then by the same argument we find $u$ satisfying $\sign(\wop[u][u]) = - \sign(\hat \calN_k)$.
\end{proof}

%% file: Lp_embeddings.tex

\section{Embeddings}\label{sec:embeddings}

We investigate the embeddings $\wdom \embeds L^p(\R \times \T)$, and discuss their boundedness in \cref{prop:Lp_embedding} as well as a concentration-compactness result in \cref{prop:Lions}. We use some ideas and results from \cite{henninger_ohrem_reichel}. 

Let us fix some notation. 
Recall that the potential $g_0$ (and therefore also $V$) is $X^+$-periodic on $[R^+, \infty)$ as well as $X^-$-periodic on $(-\infty, R^-]$:
In case \ref{ass:localized_nonlin} this is part of the assumption, and in \ref{ass:perturbed_nonlin} $g_0$ is periodic so we may choose $R^\pm$ arbitrarily and $X^+ = X^-$.
We denote by $V^+$ the $X^+$-periodic extension of $V \vert_{[R^+, \infty)}$ to $\R$ and similarly by $V^-$ the $X^-$-periodic extension of $V \vert_{(-\infty, R^-]}$ to $\R$.
To improve readability, in the following we use the $\pm$ symbol.
Statements involving such double symbols should be read always using the top, or always using the bottom symbol.
We define the periodic Sturm-Liouville operators
\begin{align*}
	L^\pm \coloneqq - \frac{1}{V^\pm(x)} \partial_x^2.
\end{align*}
According to Floquet-Bloch theory (cf. \cite{plum_floquet}), for the spectra of $L^\pm \colon H^2(\R) \to L^2(\R)$ we have
\begin{align*}
	\sigma(L^\pm) = \bigcup_{n \in \N} I_n^\pm
\end{align*}
where $I_n^\pm$ are compact intervals with $\min I_n^\pm \xrightarrow{n\to\infty} \infty$, called \emph{spectral bands}. We assume that they are enumerated in the standard way for Floquet-Block theory: $I_n^\pm$ are increasing, i.e., $\min I_{n + 1}^\pm \geq \max I_n^\pm$, and the boundary points $\set{\min I_n^\pm, \max I_n^\pm \colon n \in \N}$ consist precisely of those $\lambda \in \R$ where $L^\pm f = \lambda f$ admits nonzero $X^\pm$-periodic or $X^\pm$-antiperiodic solutions. 

The operators $L^\pm$ are useful in the study of $L$ together with information on its point spectrum. For example, $\sigma(L) = \sigma_p(L) \cup \sigma_\ess(L) = \sigma_p(L) \cup \sigma(L^+) \cup \sigma(L^-)$ holds (cf. \cite[Lemma~B.1]{henninger_ohrem_reichel}). Information on $L$ allows us to better understand $\wdom$ and characterize its embeddings. We begin with the following sufficient condition for boundedness of the $L^p$-embeddings.

\begin{lemma}\label{lem:Lp_continuity_condition}
	Let $p \in (2, \infty]$, $s \coloneqq \frac{p}{p -2}$ with
	\begin{align*}
		C \coloneqq \sum_{k \in \rfreq} \Abs{k^2 \hat \calN_k}^s \left( \sum_{n \in \N} \dist(\omega^2 k^2, I_n^+)^{-s} + \sum_{n \in \N} \dist(\omega^2 k^2, I_n^-)^{-s} + \sum_{\lambda \in \sigma_p(L)} \abs{\omega^2 k^2 - \lambda}^{-s} \right) < \infty.
	\end{align*}
	Then the embedding $\wdom \embeds L^p(\R \times \T)$ is continuous.
\end{lemma}
\begin{proof}
	We consider the isometry 
	\begin{align*}
		E \colon \wdom \to L^2_V(\R \times \T; \R), 
		~ 
		u \mapsto \sum_{k \in \rfreq} T^{-1}\Bigl[\abs{\frac{\lambda - \omega^2 k^2}{\omega^2 k^2 \hat \calN_k}}^{\frac12} T[\hat u_k](\lambda)\Bigr](x) e_k(t)
	\end{align*} 
	and the family of operators
	\begin{align*}
		\iota_\theta \colon L^2_V(\R \times \T; \R)
		\to L^{\frac{2}{1 - \theta}}(\R \times \T; \R), 
		u \mapsto \sum_{k \in \rfreq} T^{-1}\Bigl[\abs{\frac{\lambda - \omega^2 k^2}{\omega^2 k^2 \hat \calN_k}}^{- \frac{\theta s}{2}} T[\hat u_k](\lambda)\Bigr](x) e_k(t)
	\end{align*}
	where $\theta$ varies over $[0, 1]$. Then $\iota_0 = \mathrm{Id} \colon L^2(\R \times \T; \R) \to L^2(\R \times \T; \R)$ is bounded, and boundedness of $\iota_1$ follows from $C < \infty$ as in \cite[Lemma~3.26]{henninger_ohrem_reichel}. 
	For this, we note that the enumeration of the spectral bands $I_n^\pm$ used in \cite{henninger_ohrem_reichel} coincides with ours.

	By interpolation, each $\iota_\theta$ is bounded. Setting $\theta = \frac1s$, it follows that 
	\begin{align*}
		\iota_\theta E = \mathrm{Id} \colon \wdom \to L^{\frac{2}{1 - \theta}}(\R \times \T; \R) = L^p(\R \times \T; \R)
	\end{align*}
	is bounded.
\end{proof}

We are now ready to state our central embedding result, where we use information on the spectrum of $L$ given in assumption \ref{ass:spec}, in particular the existence and size of the spectral gaps about $\omega^2 k^2$ for $k \in \rfreq$ as well as estimates on the point spectrum, to verify the condition in \cref{lem:Lp_continuity_condition}.

\begin{proposition}\label{prop:Lp_embedding}
	Let $p \in [2, p^\star)$ with $p^\star \coloneqq \min\set{\frac{2 \beta}{\beta - 1}, \frac{4 \beta}{2 \beta - \alpha - \gamma + 2}}$. Then the embedding $\wdom \embeds L^p(\R \times \T)$ is continuous and locally compact. 

	In the above quotients we set $\frac{a}{b} = \infty$ for $b \leq 0$. Note that $p^\star > 4$ holds by assumption \ref{ass:L4_embed}.
\end{proposition}

\begin{proof}
	We only consider $p > 2$ and begin by showing continuity of the embedding. First we show that the sum corresponding to the point spectrum
	\begin{align}\label{eq:loc:point_spec_double_sum}
		\sum_{k \in \rfreq} \Abs{k^2 \hat \calN_k}^s \sum_{\lambda \in \sigma_p(L)} \abs{\omega^2 k^2 - \lambda}^{-s}
	\end{align}
	with $s \coloneqq \frac{p}{p - 2}$ is finite. Recall that $\Abs{\hat \calN_k} \lesssim \abs{k}^{-\alpha}$ due to assumption \ref{ass:calN}. We use the estimate
	\begin{align*}
		\sum_{k = m}^n k^r 
		\eqsim_r \int_m^n k^r 
		= \frac{n^{r + 1} - m^{r + 1}}{r + 1}
		\lesssim_r n^{r + 1} + m^{r + 1}
	\end{align*}
	on integer sums with $r \in \R \setminus \set{-1}$ and $m, n \in \N$ with $m < n$. 
	To keep notation simple, below we assume that we are in the generic case $r \neq -1$. 
	For $r = -1$ we can use $\sum_{k = m}^n k^{-1} \lesssim_\eps n^\eps$ instead, which leads to the same results provided $\eps$ is chosen sufficiently small. 

	We use separate estimates for \eqref{eq:loc:point_spec_double_sum} in the three cases $\omega^2 k^2 \ll \lambda$, $\omega^2 k^2 \approx \lambda$, and $\omega^2 k^2 \gg \lambda$. First, we calculate
	\begin{align*}
		\sum_{\lambda \in \sigma_p(L)} \sum_{\substack{k \in \rfreq \\ \omega^2 k^2 < \frac12 \lambda}} \abs{k}^{(2 - \alpha) s} \abs{\omega^2 k^2 - \lambda}^{-s}
		&\eqsim \sum_{\lambda \in \sigma_p(L)} \sum_{\substack{k \in \rfreq \\ \omega^2 k^2 < \frac12 \lambda}} \abs{k}^{(2 - \alpha) s} \lambda^{-s}
		\\* &\lesssim \sum_{\lambda \in \sigma_p(L)} \Bigl(1 + \lambda^\frac{1 + (2 - \alpha)s}{2}\Bigr) \lambda^{-s}.
	\end{align*}
	Second, we have 
	\begin{align*}
		\sum_{\lambda \in \sigma_p(L)} \sum_{\substack{k \in \rfreq \\ \omega^2 k^2 \geq 2 \lambda}} \abs{k}^{(2 - \alpha) s} \abs{\omega^2 k^2 - \lambda}^{-s}
		&\eqsim \sum_{\lambda \in \sigma_p(L)} \sum_{\substack{k \in \rfreq \\ \omega^2 k^2 \geq 2 \lambda}} \abs{k}^{(2 - \alpha) s} \abs{k}^{-2 s}
		\\* &\lesssim \sum_{\lambda \in \sigma_p(L)} \lambda^{\frac{1 - \alpha s}{2}}
	\end{align*}
	For the third and last sum we use $\abs{\omega^2 k^2 - \lambda} = (\omega \abs{k} + \sqrt{\lambda}) \abs{\omega \abs{k} - \sqrt{\lambda}} \geq \sqrt{\lambda} \omega n$ where $n = \floor{\Abs{\abs{k} - \frac{1}{\omega} \sqrt{\lambda}}}$. Observe that each $n \in \N_0$ is attained as a value for at most four $k$. 
	By \ref{ass:spec} we also have $\abs{\omega^2 k^2 - \lambda} \geq \delta \abs{k}^\gamma$, which we use instead when $n = 0$. Therefore 
	\begin{align*}
		\sum_{\lambda \in \sigma_p(L)} \sum_{\substack{k \in \rfreq \\ \frac12 \lambda \leq \omega^2 k^2 < 2 \lambda}} \abs{k}^{(2 - \alpha) s} \abs{\omega^2 k^2 - \lambda}^{-s}
		&\eqsim\sum_{\lambda \in \sigma_p(L)} \lambda^{\frac{(2 - \alpha) s}{2}} \sum_{\substack{k \in \rfreq \\ \frac12 \lambda \leq \omega^2 k^2 < 2 \lambda}} \abs{\omega^2 k^2 - \lambda}^{-s}
		\\* &\lesssim \sum_{\lambda \in \sigma_p(L)} \lambda^{\frac{(2 - \alpha) s}{2}}
		\left( \delta \lambda^{-\frac{\gamma s}{2}} + \sum_{n = 1}^{\infty} \lambda^{-\frac{s}{2}}n^{-s} \right)
	\end{align*}
	By \ref{ass:spec} all the above sums are finite provided 
	\begin{align}\label{eq:loc:Lp_cond}
		\min\set{s, \frac{\alpha s - 1}{2}, \frac{(\alpha + \gamma - 2) s}{2}, \frac{(\alpha - 1) s}{2}} > \beta.
	\end{align}
	Using $\frac{\alpha s - 1}{2} > \frac{(\alpha - 1) s}{2} \geq \frac{(\alpha + \gamma - 2)s}{2}$, a direct calculation shows that \eqref{eq:loc:Lp_cond} holds for $p < p^\star$.

	For the remaining sums appearing in the constant $C$ of \cref{lem:Lp_continuity_condition}, we first estimate
	\begin{align}\label{eq:loc:42}
		\sum_{n \in \N} \dist(\omega^2 k^2, I_n^\pm)^{-s} 
		\leq \sum_{\lambda \in S^\pm} \abs{\omega^2 k^2 - \lambda}^{-s} 
	\end{align}
	with $S^\pm \coloneqq \set{\bd I_n^\pm \colon n \in \N} \setminus \set0$. By the proof of \cite[Theorem~3.27]{henninger_ohrem_reichel} the spectral bands $I_n^\pm$ grow quadratically, so that $\sum_{\lambda \in S^\pm} \lambda^{-\frac12 - \eps} < \infty$ for all $\eps > 0$. We can therefore estimate the right-hand side of \eqref{eq:loc:42} as we did for the point spectrum, except $\beta$ is replaced by $\frac12$ in \eqref{eq:loc:Lp_cond}. 
	As we required $\beta \geq \frac12$ in \cref{thm:main}, this generates no additional requirements. 

	Finally, local compactness of the embedding follows by a frequency cutoff approximation argument as in the proof of \cite[Theorem~3.27]{henninger_ohrem_reichel}.
\end{proof}

For considerations of regularity, we use an improved embedding result introduced next, showing that low order temporal derivatives of a function $u \in \wdom$ still lie in $L^4(\R \times \T)$.

\begin{definition}
	We define the fractional temporal derivative $\abs{\partial_t}^s f$ of a function $f \colon \T \to \C$ for $s \in \R$ (with $\hat f_0 = 0$ if $s < 0$) as the Fourier multiplier with symbol $\abs{\omega k}^s$, i.e., by $\abs{\partial_t}^s f = \F^{-1}[\abs{\omega k}^s \hat f_k]$.
\end{definition}

\begin{remark}\label{rem:small_derivative_bounded}
	As in the proof of \cref{prop:Lp_embedding} we see that
	\begin{align*}
		\abs{\partial_t}^\eps \colon \wdom \to L^p(\R \times \T; \R)
	\end{align*} 
	is bounded and locally compact for $p \in [2, p^\star_\eps)$ with $p^\star_\eps \coloneqq \min\set{\frac{2 \beta}{\beta - 1}, \frac{4 \beta}{2 \beta - \alpha + 2 \eps - \gamma + 2}}$. 
	Assumption \ref{ass:L4_embed} implies $p^\star_\eps > 4$ for sufficiently small $\eps > 0$.
\end{remark}

We now prove a variant of the concentration-compactness principle of Lions.   

\begin{proposition}\label{prop:Lions}
	Let $p \in [2, p^\star)$ where $p^\star$ is given by \cref{prop:Lp_embedding}, $r > 0$, $w \colon \R \to [0, \infty)$ be a bounded and measurable weight function, and $(u_n)$ be a bounded sequence in $\wdom$ with 
	\begin{align}\label{eq:loc:evanescence_cond}
		\sup_{x \in \R} \int_{[x - r, x + r] \times \T} \abs{u_n}^p \der[w(x)] {(x,t)} 
		\xrightarrow{n \to \infty} 0.
	\end{align}
	Then $u_n \xrightarrow{n \to \infty} 0$ in $L^q_w(\R \times \T) = L^q(\R \times \T; w(x) \Der (x, t))$ for all $q \in (2, p^\star)$.
\end{proposition}

\begin{proof}
	\textit{Part 1:}
	By Hölder's inequality it suffices to give the proof for $p = 2$ and one $q \in (2, p^\star)$, which we choose later. Inspired by \cite{maier_reichel_schneider}, we consider an auxiliary Hilbert space $H$ defined by 
	\begin{align*}
		H &= \set{u \in L^2(\R \times \T; \R) \colon \hat u_k = 0\text{ for } k \in \Z \setminus \rfreq, \norm{u}_H < \infty},
		\\ \norm{u}_H^2 &= \sum_{k \in \rfreq} \abs{k}^{\alpha + \gamma - 4} \int_\R \abs{\hat u_k'}^2  + V(x) k^2\abs{\hat u_k}^2 \der x.
	\end{align*}
	The $H$-norm is local in $x$ which will allow us to get additional information on the desired embedding by considering $\wdom \embeds H \embeds L^q_w(\R \times \T)$.

	Using assumption~\ref{ass:calN}, \cref{lem:basic_calc_properties} and the estimate $\abs{\lambda - \omega^2 k^2} \gtrsim \abs{k}^{\gamma} + \lambda \abs{k}^{\gamma - 2}$ (see \cref{rem:abs_symbol_estimate}) we find
	\begin{align*}
		\norm{u}_\wdom^2
		&\gtrsim \sum_{k \in \rfreq} \abs{k}^{\alpha - 2} 
		\int_\R \abs{\lambda - \omega^2 k^2} T_i[\hat u_k](\lambda) \overline{T_j[\hat u_k](\lambda)} \der \mu_{ij}(\lambda)
		\\ &\gtrsim \sum_{k \in \rfreq} \abs{k}^{\alpha + \gamma - 4} \int_\R \left(k^2 + \lambda\right) T_i[\hat u_k](\lambda) \overline{T_j[\hat u_k](\lambda)} \der \mu_{ij}(\lambda)
		\\ &= \sum_{k \in \rfreq} \abs{k}^{\alpha + \gamma - 4} \int_\R \abs{\hat u_k'}^2  + V(x) k^2\abs{\hat u_k}^2 \der x
		= \norm{u}_H^2,
	\end{align*}
	so that $\wdom \embeds H$ is bounded.

	\textit{Part 2:} We now consider the embedding $H \embeds L^q(\R \times \T)$. For this, let $I \subseteq \R$ be an interval of length $2 r$. On $I$, let $\varphi_n(x) = \frac{1}{\sqrt{2 r}} \ee^{\ii n \frac{\pi}{r} x}$ and define the spatial Fourier transform $F[\phi]$ of $\phi \colon I \to \C$ by
	\begin{align*}
		F_n[\phi] = \int_{I} \phi(x) \overline{\varphi_n(x)} \der x
		\qquad\text{for}\quad
		n \in \Z.
	\end{align*}
	Fix some $s > 8$ and define $q > 2 > q'$ by $\frac{1}{q'} = 1 - \frac{1}{q} = \frac{1}{2} + \frac{1}{s}$.
	We calculate 
	\begin{align*}
		\MoveEqLeft \norm{ \left(\abs{k}^{\alpha + \gamma - 4} (n^2 + k^2)\right)^{-\frac12}}_{\ell^s(\Z \times \rfreq)}^s 
		= \sum_{k \in \rfreq} \abs{k}^{-\frac{(\alpha + \gamma - 4) s}{2}} \sum_{n \in \Z} (n^2 + k^2)^{-\frac{s}{2}}
		\\ &\lesssim \sum_{k \in \rfreq} \abs{k}^{-\frac{(\alpha + \gamma - 4) s}{2}} \cdot \abs{k}^{1-s}
		\leq \sum_{k \in \rfreq} \abs{k}^{1-\frac{\beta s}{2}}
		< \infty,
	\end{align*}
	where we used $\alpha + \gamma - 2 > \beta \geq \frac{1}{2}$. From this, for $u \in H$ we obtain
	\begin{align*}
		\norm{u}_{L^q_w(I \times \T)}^2
		& \lesssim \norm{u}_{L^q(I \times \T)}^2
		\lesssim \norm{F_n[\hat u_k]}_{\ell^{q'}(\Z \times \rfreq)}^2
		\\ &\leq \norm{ \left(\abs{k}^{\alpha + \gamma - 4} (n^2 + k^2)\right)^{\frac12} F_n[\hat u_k]}_{\ell^2(\Z \times \rfreq)}^2 \norm{ \left(\abs{k}^{\alpha + \gamma - 4} (n^2 + k^2)\right)^{-\frac12}}_{\ell^s(\Z \times \rfreq)}^2
		\\ &\lesssim \sum_{k \in \rfreq} \abs{k}^{\alpha + \gamma - 4} \sum_{n \in \Z} \left(n^2 + k^2 \right)\abs{F_n[\hat u_k]}^2 \der x
		\\ &\lesssim \sum_{k \in \rfreq} \abs{k}^{\alpha + \gamma - 4} \int_I \abs{\hat u_k'}^2  + V(x) k^2\abs{\hat u_k}^2 \der x.
	\end{align*}
	We now choose intervals $I_j \coloneqq [(2 j - 1) r, (2 j + 1) r]$, and define the norm
	\begin{align*}
		\norm{u}_{\ell^p L^q_w} \coloneqq \norm{\left( \norm{u}_{L^q_w(I_j \times \T)} \right)_j}_{\ell^p(\Z)}.
	\end{align*}
	Then, using the above for $I = I_j$ and summing over $j$ we obtain
	\begin{align*}
		\norm{u}_{\ell^2 L^q_w} \lesssim \norm{u}_H \lesssim \norm{u}_\wdom
	\end{align*}
	By Hölder interpolation we have
	\begin{align*}
		\norm{u_n}_{\ell^{p_\theta} L^{q_\theta}_w}
		\leq \norm{u_n}_{\ell^{2} L^{q}_w}^{\theta}
		\norm{u_n}_{\ell^{\infty} L^{2}_w}^{1 - \theta}
	\end{align*}
	with $\frac{1}{p_\theta} = \frac{\theta}{2} + \frac{1 - \theta}{\infty}$, $\frac{1}{q_\theta} = \frac{\theta}{q} + \frac{1 - \theta}{2}$ for all $\theta \in [0, 1]$. Now fix $\theta \in (0, 1)$ to be the unique solution to $p_\theta = q_\theta$, and let $(u_n)$ be a bounded sequence in $\wdom$ satisfying \eqref{eq:loc:evanescence_cond} with $p = 2$, i.e., $\norm{u_n}_{\ell^\infty L^2_w} \to 0$ as $n \to \infty$. This shows 
	\begin{align*}
		\norm{u_n}_{L^{q_\theta}_w(\R \times \T)} 
		= \norm{u_n}_{\ell^{p_\theta}L^{q_\theta}_w} 
		\lesssim \norm{u_n}_\wdom^{\theta}
		\norm{u_n}_{\ell^{\infty} L^{2}}^{1 - \theta} 
		\to 0
	\end{align*}
	as $n \to \infty$, completing the proof.
\end{proof}

%% file: dual_method.tex

\section{The dual problem}\label{sec:dual_problem}

We solve \eqref{eq:variational_wave} variationally with the dual variational method. 
Using that $h$ is positive and bounded by \ref{ass:h}, we formally substitute $v \coloneqq h^{\frac34} u^3$ in \eqref{eq:variational_wave} and multiply with $h^{\frac14} \wop^{-1}$ to obtain the dual problem
\begin{align*}
	\wop u - h \rproj[u^3] = 0
	\iff 
	\wop h^{-\frac14} v^{\frac13} - h^{\frac14} \rproj[v] = 0
	\iff 
	v^{\frac13} - h^{\frac14} \wop^{-1} h^{\frac14} \rproj[v] = 0
\end{align*}
where $v^{\frac13}$ denotes the real cube root of the real-valued function $v$. We abbreviate the weighted inverse operator by $\iwop = h^{\frac14} \wop^{-1} h^{\frac14} \rproj$, and thus consider the problem
\begin{align}\label{eq:dual_wave}
	v^{\frac13} - \iwop v = 0.
\end{align}
Having solved \eqref{eq:dual_wave}, we can formally recover a solution $u$ of \eqref{eq:variational_wave} by setting 
\begin{align*}
	u = h^{-\frac14} v^{\frac13} = \wop^{-1} h^{\frac14} \rproj[v].
\end{align*}

\begin{remark}\label{rem:negative_h}
	As stated in \cref{rem:on_main_thm}, we can also consider \eqref{eq:variational_wave} for negative $h$, for which the dual problem is given by
	\begin{align*}
		\wop u - h \rproj[u^3] = 0
		\iff 
		(-\wop) u - (-h) \rproj[u^3] = 0
		\iff
		v^{\frac13} - (-h)^{\frac14} (-\wop)^{-1} (-h)^{\frac14} \rproj[v] = 0
	\end{align*}
	with $v \coloneqq (-h)^{\frac34} u^3$.
	The properties of $\wop$ that we use below (symmetry, invertibility, indefiniteness) are also satisfied by $-\wop$. Therefore our results on \eqref{eq:variational_wave} and \eqref{eq:dual_wave} can be transferred to the negative case by the substitution $(h, \wop) \leadsto (-h, -\wop)$.
\end{remark}

Next we properly define the operator $\iwop$.
\begin{definition}\label{def:weak_dual_solution}
	Let $\iota \colon \wdom \embeds L^4(\R \times \T; \R)$ be the bounded embedding of \cref{prop:Lp_embedding} and $\iota'$ be its adjoint.
	Then, using that $\wop \colon \wdom \to \wdom'$ is invertible according to \cref{lem:wop_defined}, we define the $h$-weighted inverse by
	\begin{align*}
		\iwop \coloneqq h^{\frac14} \iota \wop^{-1} \iota' h^{\frac14} \colon L^{\frac43}(\R \times \T; \R) \to L^4(\R \times \T; \R)
	\end{align*}
	Next, we call a function $v$ a \emph{solution} to \eqref{eq:dual_wave} if $v \in L^{\frac43}(\R \times \T; \R)$ is a critical point of the energy functional
	\begin{align*}
		J \colon L^{\frac43}(\R \times \T; \R) \to \R, ~ J(v) \coloneqq \int_{\R \times \T} \frac{3}{4} \abs{v}^{\frac43} - \frac12 \iwop v \cdot v \der (x, t),
	\end{align*}
	or equivalently if $v^{\frac13} - \iwop v = 0$ in $L^4(\R \times \T; \R)$.
\end{definition}

\begin{remark}
	In \cref{def:weak_dual_solution}, note that $\rproj \iota = \iota$ holds by definition of $\wdom$, hence we can omit $\rproj$ in the definition of $\iwop$. By \cref{prop:Lp_embedding} the map $\iwop$ is locally compact.
\end{remark}

In \cref{thm:existence} we show that there exists a nonzero solution to the dual problem \eqref{eq:dual_wave}. More precisely, we show that there exists a ground state as defined below.

\begin{definition}
	We call the energy level 
	\begin{align*}
		\elgs \coloneqq \inf_{\substack{v \in L^{\nicefrac43}(\R \times \T; \R) \setminus \set{0} \\ J'(v) = 0}} J(v)
	\end{align*}
	the \emph{ground state energy level}, and any nonzero critical point $v$ of $J$ with $J(v) =  \elgs$ is called a \emph{ground state}.
\end{definition}

\begin{remark}
	The substitution $v = h^{\frac34} u^3$, which gives a one-to-one correspondence between solutions $u$ of \eqref{eq:variational_wave} and solutions $v$ to \eqref{eq:dual_wave}, also links the ground states of the two problems. Indeed, if $u$ is a solution of \eqref{eq:variational_wave}, which is the Euler-Lagrange equation of 
	\begin{align*}
		\tilde J(u) = \int_{\R \times \T} \tfrac12 \wop u \cdot u - \tfrac14 h u^4 \der (x, t),
	\end{align*}
	we have
	\begin{align*}
		\tilde J(u) = \tilde J(u) - \tfrac12 \tilde J'(u)[u]
		= \tfrac14 \int_{\R \times \T} h u^4 \der(x, t)
		= \tfrac14 \int_{\R \times \T} \abs{v}^{\frac43} \der (x, t)
		= J(v) - \tfrac12 J'(v)[v] = J(v).
	\end{align*}
\end{remark}

To show existence of ground states, we first use the mountain pass method to obtain a Palais-Smale sequence for $J$. 

\begin{definition}
	We call a sequence $(v_n)$ in $L^{\frac43}(\R\times\T; \R)$ a \emph{Palais-Smale sequence for $J$ at level $\el \in \R$} if $J(v_n) \to \el$ in $\R$ and $J'(v_n) \to 0$ in $\wdom'$ as $n \to \infty$.
\end{definition}

\begin{proposition}\label{prop:mountain_pass}
	There exists $v \in L^{\frac43}(\R \times \T; \R) \setminus \set{0}$ with $J(v) \leq 0$. For such $v$ we define the mountain pass energy level by
	\begin{align*}
		\elmp 
		\coloneqq \elmp(v) 
		\coloneqq \inf_{\substack{\gamma \in C([0, 1]; L^{\frac43}(\R \times \T; \R)) \\ \gamma(0) = 0, \gamma(1) = v}} \, \sup_{t \in [0, 1]} J(\gamma(t)).
	\end{align*} 
	Then $\elmp > 0$ and there exists a Palais-Smale sequence for $J$ at level $\elmp$.
\end{proposition}
\begin{proof}
	We first show that there exists $v \in L^{\frac43}(\R \times \T; \R)$ with $\int_{\R\times\T} \iwop v \cdot v \der (x, t) > 0$. Assume for a contradiction that
	\begin{align*}
		\int_{\R \times \T} \iwop v \cdot v \der (x, t) 
		= \int_{\R \times \T} \iota \wop^{-1} \iota' h^{\frac14} v \cdot h^{\frac14} v \der (x, t)
		\leq 0
	\end{align*}
	for all $v \in L^{\frac43}(\R \times \T; \R)$. As $h$ vanishes almost nowhere, $h^{\frac14} L^{\frac43}(\R \times \T) \subseteq L^{\frac43}(\R \times \T)$ is dense. By approximation it follows that
	\begin{align*}
		\int_{\R \times \T} \iota \wop^{-1} \iota' v \cdot v \der (x, t)
		\leq 0
	\end{align*} 
	holds for all $v \in L^{\frac43}(\R \times \T; \R)$. Next, let $\varphi$ lie in the dense subset $\wdense \subseteq \wdom$ of \cref{lem:dense}.
	Then $v \coloneqq \wop \varphi \in L^\infty(\R \times \T)$ is compactly supported, and we have
	\begin{align*}
		\int_{\R \times \T} \iota \wop^{-1} \iota' v \cdot v \der (x, t)
		= \wop[\varphi][\varphi] 
		\leq 0.
	\end{align*}
	Using density of $\wdense \subseteq \wdom$ we conclude that $\wop$ is negative semidefinite, contradicting \cref{lem:indefinite}. 

	So our assumption was false, and therefore we find $v \in L^{\frac43}(\R \times \T; \R)$ with $\int_{\R\times\T} \iwop v \cdot v \der (x, t) > 0$.
	This implies $J(s v) \leq 0$ for sufficiently large $s$.

	Now let $0 < r < \Bigl( \frac{3}{2 \norm{\iwop}} \Bigr)^{\frac32}$ and $\tilde v \in L^{\frac43}(\R \times \T)$ with $\norm{\tilde v}_{\frac43} = r$. Then
	\begin{align*}
		J(\tilde v) \geq \frac34 r^{\frac43} - \frac12 \norm{\iwop} r^2 > 0
	\end{align*}
	and therefore also $\elmp \geq \frac34 r^{\frac43} - \frac12 \norm{\iwop} r^2 > 0$.
	By the mountain pass theorem (cf. Theorem~6.1, Theorem~3.4, and Remark~3.5 in \cite[Chapter~II]{struwe}) there exists a Palais-Smale sequence $(v_n)$ for $J$ at level $\elmp$.
\end{proof}

\begin{remark}\label{rem:ps_bounded}
	Any Palais-Smale sequence for $J$ is bounded. Indeed, if $(v_n)$ is a Palais-Smale sequence at level $\el$, then 
	\begin{align*}
		2 \el + \landauo(1) + \landauo(\norm{v_n}_{\frac43}) 
		= 2 J(v_n) - J'(v_n)[v_n]
		= \frac{1}{2} \norm{v_n}_{\frac43}^{\frac43},
	\end{align*}
	shows that $(v_n)$ is bounded and moreover that $\norm{v_n}_{\frac43} \to \left( 4 \el \right)^{\frac34}$ as $n \to \infty$.
\end{remark}

Next we show an existence result for the dual problem.

\begin{theorem}\label{thm:existence}
	There exists a ground state of \eqref{eq:dual_wave}.
\end{theorem}

The proof of \cref{thm:existence} differs depending on the choice of assumption: \ref{ass:localized_nonlin} or \ref{ass:perturbed_nonlin}. 
The case \ref{ass:localized_nonlin} is simpler since $J$ satisfies the Palais-Smale condition, and the proof is carried out in \cref{prop:existence:localized}.
Case \ref{ass:perturbed_nonlin} is investigated in \cref{lem:existence:periodic} for purely periodic coefficients $g_0, g_1, h$ using the concentration-compactness principle \cref{prop:Lions}, and in \cref{prop:existence:perturbed} for the general case using energy comparison arguments.

\begin{proposition}\label{prop:existence:localized}
	Assume \ref{ass:first}--\ref{ass:last_nogeom_u} and \ref{ass:localized_nonlin}. Then there exists a ground state of \eqref{eq:dual_wave}.
\end{proposition}
\begin{proof}
	Let $(v_n)$ be the Palais-Smale sequence from \cref{prop:mountain_pass}. 
	By \cref{rem:ps_bounded}, up to a subsequence which we again label by $v_n$, there exists $v \in L^{\frac43}(\R \times \T; \R)$ with $v_n \wto v$ in $L^{\frac43}(\R \times \T)$.

	By \cref{prop:Lp_embedding} the embedding $\iota \colon \wdom \embeds L^4(\R \times \T; \R)$ is locally compact. Then $h^{\frac14} \iota$ is compact since $h$ decays to $0$ at $\pm\infty$ by assumption~\ref{ass:localized_nonlin}, and in particular $\iwop$ is compact and thus $\iwop v_n \to \iwop v$ in $L^4$. Using
	\begin{align*}
		J'(v_n) = v_n^{\frac13} - \iwop v_n \to 0\text{ in }L^4
	\end{align*}
	we see that $v_n^{\frac13}$ converges to $\iwop v$ in $L^4$, which implies $v_n \to (\iwop v)^3$ in $L^{\frac43}$. 
	
	Since also $v_n \wto v$, we have $v = (\iwop v)^3$ and $v_n \to v$ in $L^{\frac43}$, so $v^{\frac13} - \iwop v = 0$ in $L^4$. By continuity of $J$ we have $J(v_n) \to J(v)$, i.e., $J(v) = \elmp$.

	Thus far, we have shown existence of a nonzero critical point of $J$, and thus $\elgs \neq \infty$.
	Now let $(v_n)$ be a sequence of critical points of $J$ with $J(v_n) \to \elgs$. Then the above arguments shows $v_n \to v$ in $L^{\frac43}$ up to a subsequence, and hence $v$ is a ground state.
\end{proof}

\begin{lemma}\label{lem:existence:periodic}
	Assume \ref{ass:first}--\ref{ass:last_nogeom_u} and \ref{ass:perturbed_nonlin} with $h^\loc = 0$. Then there exists a ground state of \eqref{eq:dual_wave}.
\end{lemma}
\begin{proof}
	\textit{Part 1:}
	Denote the common period of $h = h^\per, g_0, g_1$ by $X$, so $V$ is also $X$-periodic by its definition \eqref{eq:def:V}. Let us investigate the shift $\tau$ in $x$ by $X$, i.e., $\tau[f](x) = f(x - X)$. With the spectral transform $T$ and fundamental solution $\Psi$ given in \cref{thm:functional_calc}, we calculate for compactly supported $f \in L^2(\R; \C)$
	\begin{align*}
		T[\tau f](\lambda) 
		= \int_\R f(x - X) \Psi(x; \lambda) \der x
		= \int_\R f(x) \Psi(x + X; \lambda) \der x
		= M(\lambda) T[f](\lambda)
	\end{align*}
	where the matrix $M(\lambda) \in \R^{2 \times 2}$ is given by
	\begin{align*}
		M(\lambda) \Psi(x; \lambda) = \Psi(x + X; \lambda).
	\end{align*}
	for all $x \in \R$, and it exists since $\Psi(\impvar; \lambda)$ solves an $X$-periodic differential equation.
	As $\tau$ is an isometric isomorphism on $L^2_V(\R; \C)$, multiplication with $M(\lambda)$ is an isometric isomorphism of $L^2(\mu)$. For $u \in \wdom$ we have
	\begin{align*}
		\norm{\tau u}_\wdom^2 
		&= \sum_{k \in \rfreq} \norm{ \abs{\frac{\lambda - \omega^2 k^2}{\omega^2 k^2 \hat \calN_k}}^{\frac12} T[\tau \hat u_k](\lambda)}_{L^2(\mu)}^2
		\\ &= \sum_{k \in \rfreq} \norm{ M(\lambda) \abs{\frac{\lambda - \omega^2 k^2}{\omega^2 k^2 \hat \calN_k}}^{\frac12} T[\hat u_k](\lambda)}_{L^2(\mu)}^2
		\\ &= \sum_{k \in \rfreq} \norm{ \abs{\frac{\lambda - \omega^2 k^2}{\omega^2 k^2 \hat \calN_k}}^{\frac12} T[\hat u_k](\lambda)}_{L^2(\mu)}^2
		= \norm{u}_\wdom^2,
	\end{align*}
	i.e., $\tau$ is an isometric isomorphism of $\wdom$. A similar calculation shows $\wop_0[\tau u][\tau u] = \wop_0[u][u]$, and $\wop_1[\tau u][\tau u] = \wop_1[u][u]$ follows directly from periodicity of $\calG$. Thus we have $\wop[\tau u][\tau u] = \wop[u][u]$. 
	Let us define $\tau$ on $\wdom'$ by $\tau \vert_{\wdom'} = \left(\left( \tau \vert_{\wdom} \right)^{-1}\right)'$, which is an isometric isomorphism of $\wdom'$. Then by the above, $\tau \wop = \wop \tau$, $\tau \wop^{-1} = \wop^{-1} \tau$ and $\tau \iwop = \iwop \tau$ hold.

	\textit{Part 2:} 
	Let $(v_n)$ be the Palais-Smale sequence given by \cref{prop:mountain_pass}. 
	We apply \cref{prop:Lions} with $r = X$, $p = q = 4$, $w = h$ and $u_n \coloneqq \wop^{-1} \iota' h^{\frac14} v_n = h^{-\frac14} \iwop v_n$ to obtain a sequence of points $x_n \in \R$ with 
	\begin{align*}
		\limsup_{n \to \infty} \norm{u_n}_{L^4_h([x_n - X, x_n + X] \times \T)} 
		= \limsup_{n \to \infty} \norm{\iwop v_n}_{L^4([x_n - X, x_n + X] \times \T)} 
		> 0.
	\end{align*}
	since 
	\begin{align*}
		\norm{u_n}_{L^4_h(\R \times \T)}
		= \norm{\iwop v_n}_{L^4(\R \times \T)}
		= \norm{v_n^3}_{L^4(\R \times \T)} + \landauo(1)
		\to (4 \elmp)^{\frac94} > 0
	\end{align*}
	by \cref{rem:ps_bounded}.
	
	W.l.o.g. we may assume $x_n = k_n X$ for some $k_n \in \Z$. Then, $\tilde v_n \coloneqq \tau^{k_n} v_n$ satisfies
	\begin{align*}
		\norm{\iwop v_n}_{L^{4}([x_n - X, x_n + X] \times \T)}
		= \norm{\tau^{-k_n} \iwop \tau^{k_n} v_n}_{L^{4}([x_n - X, x_n + X] \times \T)}
		= \norm{\iwop \tilde v_n}_{L^{4}([- X, X] \times \T)}.
	\end{align*}
	Now choose a subsequence, again denoted by $\tilde v_n$ such that $\tilde v_n \wto v$ in $L^{\frac43}(\R \times \T)$ and
	\begin{align*}
		\norm{\iwop \tilde v_n}_{L^4([-X, X] \times \T)} \to s > 0
	\end{align*}
	By local compactness of $\iwop$ (see \cref{prop:Lp_embedding}) 
	we have $\iwop \tilde v_n \to \iwop v$ in $L^4_\loc(\R \times \T)$ and thus $v \neq 0$.
	In addition, 
	\begin{align*}
		\int_{\R \times \T} \left(\tilde v_n^{\frac13} - \iwop \tilde v_n \right)\varphi \der (x, t)
		= \int_{\R \times \T} \left(v_n^{\frac13} - \iwop v_n \right) \tau^{-k_n} \varphi \der (x, t)
		= 
		J'(v_n)[\tau^{-k_n} \varphi] = \landauo(\norm{\varphi}_{\frac43})
	\end{align*}
	as $n \to \infty$ for $\varphi \in L^{\frac43}(\R \times \T)$. This shows $\tilde v_n^{\frac13} \to \iwop v$ in $L^4_\loc(\R \times \T; \R)$ and thus $\tilde v_n \to (\iwop v)^3$ in $L^{\frac43}_\loc(\R \times \T; \R)$. Therefore $v = (\iwop v)^3$ holds, that is, $v$ is a critical point of $J$. In particular, we have $\elgs \neq \infty$.

	\textit{Part 3:} 
	Now let $(v_n)$ be a sequence of critical points of $J$ with $J(v_n) \to \elgs$. Arguing as in part~2, up to a subsequence and shifts $x_n = k_n X \in \R$ we have $\tilde v_n \coloneqq \tau^{k_n} v_n \wto v$, where $v$ is a nonzero critical point of $J$.
	For the energy level of $v$ we calculate
	\begin{align*}
		J(v)
		= J(v) - \frac12 J'(v)[v]
		= \frac14 \norm{v}_{\frac43}^{\frac43}
		\leq \liminf_{n \to \infty} \frac14 \norm{\tilde v_n}_{\frac43}^{\frac43}
		= \liminf_{n \to \infty} J(\tilde v_n) - \frac12 J'(\tilde v_n)[\tilde v_n]
		= \elgs.
	\end{align*}
	Since also $J(v) \geq \elgs$ by definition of the ground state energy, we see that $v$ is a ground state.
\end{proof}

\begin{proposition}\label{prop:existence:perturbed}
	Assume \ref{ass:first}--\ref{ass:last_nogeom_u} and \ref{ass:perturbed_nonlin}. Then there exists a ground state of \eqref{eq:dual_wave}.
\end{proposition}
\begin{proof}
	\textit{Part 1:}
	We consider the periodic functional 
	\begin{align*}
		J^\per(v) = \int_{\R \times \T} \tfrac34 \abs{v}^{\frac43} - \tfrac12 \iwop[h^\per] v \cdot v \der (x, t)
	\end{align*}
	on $L^{\frac43}(\R \times \T)$ which has a ground state $v^\per$ due to \cref{lem:existence:periodic}. We denote its energy by $\elgs^\per \coloneqq J^\per(u^\per)$.  Recall that $0 < h^\per \leq h^\per + h^\loc = h$ holds by assumption~\ref{ass:perturbed_nonlin}. Setting $v \coloneqq \left(\frac{h^\per}{h}\right)^{\frac14} v^\per$, we estimate
	\begin{align*}
		J(s v) &= \int_{\R \times \T} \frac{3}{4} s^{\frac43} \abs{v}^{\frac43} - \frac12 s^2 \iwop v \cdot v \der (x, t)
		\\* &= \int_{\R \times \T} \frac{3}{4} s^{\frac43} \left(\tfrac{h^\per}{h}\right)^{\frac13} \abs{v^\per}^{\frac43} - \frac12 s^2 \iwop[h^\per] v^\per \cdot v^\per \der (x, t)
		\\* &\leq J^\per(s v^\per)
	\end{align*}
	for any $s > 0$. For the particular value $s_0 \coloneqq \left(\frac{3}{2}\right)^{\frac32}$ we calculate
	\begin{align*}
		J^\per(s_0 v^\per) = \frac{27}{16} (J^\per)'(v^\per)[v^\per] = 0.
	\end{align*}
	In particular, we can estimate the mountain pass energy via
	\begin{align*}
		\elmp \coloneqq \elmp(s_0 v) \leq \max_{s \in [0, s_0]} J(s v)
		\leq \max_{s \in [0, s_0]} J^\per(s v^\per) = J^\per(v^\per) \eqqcolon \elgs^\per.
	\end{align*}

	\textit{Part 2:} Let us first consider the case $\elmp = \elgs^\per$. From the above (in)equality we conclude that $v^\per = 0$ holds almost everywhere on $\set{h \neq h^\per}$. In particular, we have $v = v^\per$, and from
	\begin{align*}
		\left(v^\per\right)^{\frac13} - \iwop[h^\per] v^\per = 0
	\end{align*}
	we see that $\iwop[h^\per] v^\per$ also vanishes on $\set{h \neq h^\per}$. Therefore
	\begin{align*}
		\left(v\right)^{\frac13} - \iwop[h] v
		= \left(v^\per\right)^{\frac13} - \iwop[h^\per]
		= 0,
	\end{align*}
	and $v = v^\per$ is a ground state of both $J$ and $J^\per$.

	\textit{Part 3:} Let us now consider the case $\elmp < \elgs^\per$. Let $v_n$ be a Palais-Smale sequence for $J$ with $J(v_n) \to \elmp$. Then up to a subsequence $v_n \wto v \in L^{\frac43}(\R \times \T)$ by \cref{rem:ps_bounded}. 
	In order to show $v \neq 0$, let us assume for a contradiction that $v = 0$ so that $\iwop v_n \to 0$ in $L^4_\mathrm{loc}$ by \cref{prop:Lp_embedding}.
	As in the proof of \cref{lem:existence:periodic} there exists a subsequence, again denoted by $v_n$, such that suitable shifts $\tilde v_n = \tau^{k_n} v_n$ weakly converge to some $0 \neq \tilde v \in L^{\frac43}(\R \times \T; \R)$. 
	From $v_n^3 = \iwop v_n + \landauo(1) \to 0$ in $L^4_\loc$ we conclude $\abs{k_n} \to \infty$.
	For compactly supported $\varphi \in L^{\frac43}(\R \times \T; \R)$ we calculate
	\begin{align*}
		&\abs{(J^\per)'(v_n)[\tau^{-k_n} \varphi] - J'(v_n)[\tau^{-k_n} \varphi]}
		\\ &\quad= \abs{\int_{\R \times \T} \left( \iwop[h] - \iwop[h^\per] \right) v_n \cdot \tau^{-k_n} \varphi \der (x, t)}
		\\ &\quad= \abs{\int_{\R \times \T} \left( \bigl(h^{\frac14} - \left(h^\per\right)^{\frac14}\bigr) \iota\wop^{-1}\iota' h^{\frac14} + \left(h^\per\right)^{\frac14} \iota\wop^{-1}\iota' \bigl(h^{\frac14} - \left(h^\per\right)^{\frac14}\bigr)\right) v_n \cdot \tau^{-k_n} \varphi \der(x, t)}
		\\ &\quad\leq \norm{\iota}^2 \norm{\wop^{-1}} \left( 
			\norm{h^{\frac14} v_n}_{\frac43} 
			\norm{\bigl(h^{\frac14} - \left(h^\per\right)^{\frac14}\bigr) \tau^{-k_n}\varphi}_{\frac43} 
			+ \norm{\bigl(h^{\frac14} - \left(h^\per\right)^{\frac14}\bigr) v_n}_{\frac43} 
			\norm{\left(h^\per\right)^{\frac14} \tau^{-k_n} \varphi}_{\frac43} 
		\right)
		\\ &\quad\lesssim 
		\norm{v_n}_{\frac43} 
		\norm{(h^\loc)^{\frac14} \tau^{-k_n} \varphi}_{\frac43} 
		+ \norm{(h^\loc)^{\frac14} v_n}_{\frac43} 
		\norm{\varphi}_{\frac43} 
		\to 0
	\end{align*}
	as $n \to \infty$ since $h^\loc$ is localized by \ref{ass:perturbed_nonlin}. Thus we have
	\begin{align*}
		\int_{\R \times \T} \tilde v_n^{\frac13} \varphi - \iwop[h^\per] \tilde v_n \cdot \varphi \der (x, t)
		= (J^\per)'(\tilde v_n)[\varphi]
		= (J^\per)'(v_n)[\tau^{-k_n} \varphi] \to 0.
	\end{align*}
	Using $\iwop[h^\per] \tilde v_n \to \iwop[h^\per] \tilde v$ in $L^4_\loc$, the above shows $\tilde v_n \to (\iwop[h^\per] \tilde v)^3$ in $L^{\frac43}_\loc$, so that $\tilde v^{\frac13} - \iwop[h^\per] \tilde v = 0$. As $\tilde v$ is a nonzero solution to the periodic problem, we have
	\begin{align*}
		\elgs^\per 
		&\leq J^\per(\tilde v)
		= J^\per(\tilde v) - 2 (J^\per)'(\tilde v)[\tilde v]
		= \frac14 \norm{\tilde v}_{\frac43}^{\frac43}
		\\ &\leq \liminf_{n \to \infty} \frac14 \norm{\tilde v_n}_{\frac43}^{\frac43}
		= \liminf_{n \to \infty} J(v_n) - 2 J'(v_n)[v_n] 
		= \elmp,
	\end{align*}
	a contradiction.
	
	So we have shown $v_n \wto v \neq 0$. By testing with compactly supported $\varphi$ as above, we see that $v$ solves $v^{\frac13} - \iwop v = 0$ and $J(v) \leq \liminf_{n \to \infty} J(v_n) = \elmp$ holds. 
	So we have found a critical point $v$ with $J(v) \leq \elmp < \elgs^\per$. 
	
	\textit{Part 4:}
	By parts 2 or 3 we have $\elgs \leq \elgs^\per$. We then choose $v_n$ to be a Palais-Smale sequence with $J(v_n) \to \elgs$. Repeating the arguments of parts~2~or~3 we obtain a nonzero critical point $v$ of $J$ satisfying $J(v) \leq \elgs$. By definition of the ground state energy, $v$ is a ground state.
\end{proof}

Lastly, we discuss a multiplicity result for solutions.
\begin{proposition}\label{prop:multiplicity}
	Assume in addition that the set $\rfreq$ is infinite. Then there exist infinitely many solutions of \eqref{eq:dual_wave} that are not spatiotemporal shifts of each other.
\end{proposition}
\begin{proof}
	For fixed $m \in \Nodd$, we seek $\frac{T}{2 m}$-antiperiodic in time solutions to \eqref{eq:dual_wave}, or equivalently functions with frequency support on $\rfreq \cap m \Zodd$. 
	These solutions are precisely critical points of $J$ restricted to the space of $\frac{T}{2m}$-antiperiodic functions.
	
	For infinitely many $m$ we have $\rfreq \cap m \Zodd \neq \emptyset$, and for these $m$ as in \cref{thm:existence} we obtain existence of a nonzero $\frac{T}{m}$-periodic solution $v_m$ to \eqref{eq:dual_wave}. Each $v_m$ has a minimal temporal period $T_m > 0$ which is a divisor of $\frac{T}{m}$. 
	From $0 < T_m \leq \frac{T}{m}$ we see that there exist infinitely many distinct minimal periods, and the corresponding solutions $v_m$ clearly are not shifts of one another.
\end{proof}

%% file: regularity.tex

\section{Regularity}\label{sec:regularity}

In this section, we discuss differentiability and integrability properties of a solution $v$ to \eqref{eq:dual_wave} as well as for corresponding solutions $u$ to \eqref{eq:variational_wave}, $w$ to \eqref{eq:scalar_wave} and $\bfE, \bfD, \bfB, \bfH$ to Maxwell's equations.

\begin{lemma}\label{lem:regularity:u}
	Let $v \in L^{\frac43}(\R \times \T; \R)$ be a solution to \eqref{eq:dual_wave}. Then $u \coloneqq h^{-\frac14} v^{\frac13}$ lies in $\wdom$ and is a weak solution to \eqref{eq:variational_wave} with $\partial_t^l \partial_x^m u \in L^2 \cap L^\infty(\R \times \T; \R)$ for $l \in \N_0$ and $m \in\set{0,1,2}$.
\end{lemma}
\begin{proof}
	\textit{Part 1:} 
	As $v$ solves \eqref{eq:dual_wave}, the function $u$ satisfies
	\begin{align*}
		u 
		= h^{-\frac14} \iwop v
		= \iota \wop^{-1} \iota' h^{\frac14} v
		= \iota \wop^{-1} \iota' h u^3,
	\end{align*}
	so that $u = \wop^{-1} \iota' h^{\frac14} v \in \wdom$.
	Then for $\varphi \in \wdom$ we have
	\begin{align*}
		0 = \wop[u - \wop^{-1} \iota' h u^3][\varphi]
		= \wop[u][\varphi] - \int_{\R \times \T} h u^3 \varphi \der (x, t),
	\end{align*}
	i.e., $u$ is a weak solution to \eqref{eq:variational_wave}.

	\textit{Part 2:} 
	According to \cref{rem:small_derivative_bounded}, we fix $\eps > 0$ such that $\abs{\partial_t}^\eps \colon \wdom \to L^4(\R \times \T; \R)$ is bounded, therefore $\abs{\partial_t}^\eps u \in L^4(\R \times \T; \R)$ holds. 
	From the fractional Leibniz rule (cf. \cite{benyi_oh_zhao}) we obtain $\abs{\partial_t}^\eps [u^3] \in L^{\frac43}(\R \times \T; \R)$ with norm $\norm{\abs{\partial_t}^\eps [u^3]}_{\frac43} \lesssim \left( \norm{u}_4 + \norm{\abs{\partial_t}^\eps u}_4 \right)^3$. Thus
	\begin{align*}
		\abs{\partial_t}^\eps u = \abs{\partial_t}^\eps \wop^{-1} \iota' h u^3
		= \wop^{-1} \iota' h \abs{\partial_t}^\eps[u^3] \in \wdom
	\end{align*}
	holds, and from the embedding $\abs{\partial_t}^{2\eps} u \in L^4$ follows. Iterating the above argument, we get $\abs{\partial_t}^{n \eps} \in \wdom$ for all $n \in \N_0$. 
	Recall that $\wdom$ is supported on frequencies $k \in \rfreq$ and that $0 \not\in \rfreq$. Therefore $\abs{\partial_t}^s \colon \wdom \to \wdom$ is bounded for all $s \leq 0$, and the above shows $\abs{\partial_t}^s u \in \wdom$ for all $s \in \R$.

	\textit{Part 3:} From boundedness of the embedding $\wdom \embeds L^2$ (see \cref{prop:Lp_embedding}) it follows that $\abs{\partial_t}^s u \in L^2$ for all $s \in \R$. We next calculate
	\begin{align*}
		\norm{\abs{\partial_t}^s u}_\infty^2
		\leq \left( \sum_{k \in \rfreq} \abs{\omega k}^{s} \norm{\hat u_k}_\infty \right)^2
		\leq \sum_{k \in \rfreq} \frac{1}{\omega^2 k^2} \cdot \sum_{k \in \rfreq} \abs{\omega k}^{2s + 2}\norm{\hat u_k}_\infty^2
		\lesssim \sum_{k \in \rfreq} \abs{\omega k}^{2s + 2}\norm{\hat u_k}_{H^1}^2.
	\end{align*}
	Using \cref{lem:basic_calc_properties,rem:abs_symbol_estimate} and assumption~\ref{ass:calN} we further estimate
	\begin{align*}
		\sum_{k \in \rfreq} \abs{\omega k}^{2s + 2} \norm{\hat u_k}_{H^1}^2
		&\lesssim \sum_{k \in \rfreq} \abs{\omega k}^{2s + 2} \int_\R (1 + \lambda) T_i[\hat u_k] \overline{T_j[\hat u_k]} \der \mu_{ij}(\lambda)
		\\ &\lesssim \sum_{k \in \rfreq} \abs{k}^{2 s + 4 - \gamma} \int_\R \abs{\lambda - \omega^2 k^2} T_i[\hat u_k] \overline{T_j[\hat u_k]} \der \mu_{ij}(\lambda)
		\\ &\lesssim \sum_{k \in \rfreq} \abs{\omega k}^{2 s + 6 - \gamma - \alpha} \int_\R \abs{\frac{\lambda - \omega^2 k^2}{\omega^2 k^2 \hat \calN_k}} T_i[\hat u_k] \overline{T_j[\hat u_k]} \der \mu_{ij}(\lambda)
		\\ &= \norm{\abs{\partial_t}^{\frac{2s + 6 - \gamma - \alpha}{2}} u}_\wdom^2 < \infty
	\end{align*}
	for any $s \in \R$. In particular, we have $\partial_t^l u \in L^2 \cap L^\infty$ for all $l \in \N_0$.
	
	\textit{Part 4:} 
	For regularity of spatial derivatives, from \eqref{eq:scalar_wave:1:variational} we obtain the identity
	\begin{align}\label{eq:second_space_derivative_formula}
		\partial_x^2 u = V(x) \partial_t^2 u + \partial_t^2 \calG \ast u + h \partial_t^2 \calN \ast \rproj[u^3].
	\end{align}
	Observe that the right-hand side of \eqref{eq:second_space_derivative_formula} is infinitely differentiable in time with values in ${L^2 \cap L^\infty}$ since $u, u^3$ are and the convolution operators $\calG\ast, \calN \ast, \rproj$ are regularity preserving (see \cref{lem:polynomial_multiplier} below), so the same holds for $\partial_x^2 u$.
\end{proof}

\begin{lemma}\label{lem:polynomial_multiplier}
	Let $M$ be a Fourier multiplier with symbol $\hat M_k$ of at most polynomial growth, i.e., $\abs{\hat M_k} \lesssim \abs{k}^r$ for some $r > 0$ and all $k \in \Z$. Further let $p \in [1, \infty]$ and $f \colon \T \to \C$ be a function with $\partial_t^n f \in L^p(\T; \C)$ for all $n \in \N_0$ and $\hat f_0 = 0$. Then $\partial_t^n M f \in L^p(\T; \C)$ for all $n \in \N_0$.
\end{lemma}
\begin{proof}
	The symbol $(\ii \omega k)^{-\ceil{r+1}} \hat M_k$ is square-summable and therefore
	\begin{align*}
		m(t) \coloneqq \sum_{k \in \Z\setminus\set{0}} (\ii \omega k)^{-\ceil{r+1}} \hat M_k e_k(t)
	\end{align*}
	converges in $L^2(\T; \C)$. The claim follows from this and
	\begin{align*}
		\norm{\partial_t^n M f}_p
		= \norm{\partial_t^{n+ \ceil{r+1}} m \ast f}_p
		\leq \norm{m}_1 \norm{\partial_t^{n+\ceil{r+1}} f}_p.
		&\qedhere
	\end{align*}
\end{proof}

After having shown regularity of the solution $u$ to \eqref{eq:variational_wave}, we now consider the profile $w$ of the electric field. For polarization \eqref{eq:scalar_polarization:1} there is nothing to do since $w = u$. So let us discuss polarization \eqref{eq:scalar_polarization:2} where $w$ is defined by \eqref{eq:wave_reconstruction:polarization2}. 

\begin{lemma}\label{lem:regularity:w}
	Let $v \in L^{\frac43}(\R \times \T; \R)$ be a solution to \eqref{eq:dual_wave}. We define $u \coloneqq h^{-\frac14} v^{\frac13}$ as in \cref{lem:regularity:u}, and $w$ by \eqref{eq:wave_reconstruction:polarization2}. Then $\partial_t^l \partial_x^m w \in L^2 \cap L^\infty(\R \times \T; \R)$ for all $l \in \N_0$ and $m \in \set{0,1,2}$.
\end{lemma}
\begin{proof}
	\textit{Part 1:}
	By \cref{lem:regularity:u,lem:polynomial_multiplier} we already know that the functions 
	\begin{align*}
		\partial_t^l \rproj w = (\calN \ast)^{-1} \partial_t^l u
		\qquad\text{and}\qquad
		\partial_t^{l+2} (\mathrm{Id} - \rproj)[u^3]
	\end{align*}
	lie in $L^2 \cap L^\infty$ for any $l \in \N_0$. Hence it remains to show
	\begin{align*}
		\partial_t^l (\mathrm{Id} - \rproj) w 
		= \left( -\partial_x^2 + V(x) \partial_t^2 + \partial_t^2 \calG(x) \ast \right)^{-1} [- h(x) \partial_t^{l+2} (\mathrm{Id} - \rproj)[u^3]] \in L^2 \cap L^\infty,
	\end{align*}
	and in particular that $-\partial_x^2 + V(x) \partial_t^2 + \partial_t^2 \calG(x) \ast$ is invertible on suitable spaces of functions $f$ satisfying $\rproj f = 0$.
	
	\textit{Part 2:} Taking the Fourier series in time decomposes the linear operator $-\partial_x^2 + V(x) \partial_t^2 + \partial_t^2 \calG(x) \ast$ into the sequence of operators $(L_k)_{k \in \Zodd \setminus \rfreq}$ with
	\begin{align*}
		L_k \coloneqq -\partial_x^2 - \omega^2 k^2 V(x) - \omega^2 k^2 \hat \calG_k(x)
	\end{align*}
	which we split into main part $L_{k,0}$ and perturbation $L_{k,1}$ via
	\begin{align*}
		L_{k,0} \coloneqq -\partial_x^2 - \omega^2 k^2 V(x), 
		\qquad L_{k,1} \coloneqq - \omega^2 k^2 \hat \calG_k(x).
	\end{align*}

	First, for $\varphi \in H^2(\R; \C)$ using \cref{lem:basic_calc_properties,rem:abs_symbol_estimate} we calculate
	\begin{align*}
		\int_\R \abs{L_{k,0} \varphi}^2 \der[\frac{1}{V}] x
		&= \int_\R (\lambda - \omega^2 k^2)^2 T_i[\varphi] \overline{T_j[\varphi]} \der \mu_{ij}(\lambda) 
		\\ &\geq \int_\R \frac{(\tilde \delta \abs{k}^{\tilde \gamma})^2}{2} \left(1 + \frac{\lambda^2}{(\omega^2 k^2 + \delta \abs{k}^{\tilde \gamma})^2} \right) T_i[\varphi] \overline{T_j[\varphi]} \der \mu_{ii}(\lambda)
		\\ &\gtrsim \abs{k}^{2 \tilde \gamma - 4} \int_\R (1 + \lambda^2) T_i[\varphi] \overline{T_j[\varphi]} \der \mu_{ij}(\lambda)
		\\ &= \abs{k}^{2 \tilde \gamma - 4} \int_\R \Bigl(\abs{\frac{1}{V} \varphi''}^2 + \abs{\varphi}^2\Bigr) \der[V](x),
	\end{align*}
	and it follows that $L_{k,0} \colon H^2(\R; \C) \to L^2(\R; \C)$ is invertible with $\norm{L_{k,0}^{-1}} \lesssim \abs{k}^{2 - \tilde \gamma}$. For $L_{k,1}$, using \ref{ass:add_for_polarization_2} and \cref{rem:abs_symbol_estimate} we estimate
	\begin{align*}
		\int_\R \abs{L_{k,1} \varphi}^2 \der[\frac{1}{V}]{x}
		&\leq (\omega^2 k^2)^2 \left(\frac{\tilde d}{\omega^2 \abs{k}^{2 - \tilde \gamma}}\right)^2 \int_\R \abs{\varphi}^2 \der[V]{x}
		= \left( \tilde d \abs{k}^{\tilde \gamma}\right)^2 \int_\R T_i[\varphi] \overline{T_j[\varphi]} \der \mu_{ij}(\lambda)
		\\ &\leq \left(\tilde d \abs{k}^{\tilde \gamma}\right)^2 \int_\R \left(\frac{\abs{\lambda - \omega^2 k^2}}{\tilde \delta \abs{k}^{\tilde \gamma}}\right)^2 T_i[\varphi] \overline{T_j[\varphi]} \der \mu_{ij}(\lambda)
		= \frac{\tilde d^2}{\tilde \delta^2} \int_\R \abs{L_{k,0} \varphi}^2 \der[\frac{1}{V}]{x}.
	\end{align*}
	Since $\tilde d < \tilde \delta$, the Neumann series shows that $L_k = L_{k,0} + L_{k,1} \colon H^2(\R; \C) \to L^2(\R; \C; \frac{1}{V}\Der x)$ is invertible with 
	\begin{align*}
		\norm{L_k^{-1}} 
		\leq \frac{\norm{L_{k,0}^{-1}}}{1 - \norm{L_{k,1} L_{k,0}^{-1}}} 
		\leq \frac{\norm{L_{k,0}^{-1}}}{1 - \frac{\tilde d}{\tilde \delta}} 
		\lesssim \abs{k}^{2 - \tilde \gamma}.
	\end{align*}

	\textit{Part 3:} 
	We define the inverse of the linear operator by 
	\begin{align*}
		\left(-\partial_x^2 + V(x) \partial_t^2 + \partial_t^2 \calG(x) \ast\right)^{-1} \varphi 
		\coloneqq \sum_{k \in \Zodd \setminus \rfreq} L_k^{-1}[\hat \varphi_k](x) e_k(t)
	\end{align*}
	for $\varphi$ such that the series on the right-hand side converges in $L^2$.
	Let us abbreviate
	\begin{align*}
		f \coloneqq - h(x) \partial_t^{l+2} (\mathrm{Id} - \rproj)[u^3],
	\end{align*}
	which lies in $L^2 \cap L^\infty$ by \cref{lem:polynomial_multiplier} and boundedness of $h$, 
	and recall
	\begin{align*}
		\partial_t^l (\mathrm{Id} - \rproj) w 
		= \left( -\partial_x^2 + V(x) \partial_t^2 + \partial_t^2 \calG(x) \ast \right)^{-1} f 
		= \sum_{k \in \Zodd \setminus \rfreq} L_k^{-1}[\hat f_k](x) e_k(t).
	\end{align*}
	We show that this term lies in $L^2 \cap L^\infty$. First we estimate
	\begin{align*}
		\norm{\partial_t^l (\mathrm{Id} - \rproj)w}_2^2
		&= \sum_{k \in \Zodd \setminus\rfreq} \norm{L_k^{-1} \hat f_k}_2^2 
		\lesssim \sum_{k \in \Zodd \setminus\rfreq} \norm{L_k^{-1} \hat f_k}_{H^2}^2 
		\\ &\lesssim \sum_{k \in \Zodd \setminus\rfreq} \abs{\omega k}^{2 \ceil{2 - \tilde \gamma}} \norm{\hat f_k}_2^2
		= \norm{\partial_t^{\ceil{2 - \tilde \gamma}} f}_2^2 < \infty.
	\intertext{Next for $L^\infty$ we have}
		\norm{\partial_t^l (\mathrm{Id} - \rproj)w}_\infty^2
		&\leq \left( \sum_{k \in \Zodd \setminus \rfreq} \norm{L_k^{-1} \hat f_k}_\infty \right)^2
		\lesssim \sum_{k \in \Zodd \setminus \rfreq} k^2 \norm{L_k^{-1} \hat f_k}_\infty^2
		\\ &\lesssim \sum_{k \in \Zodd \setminus \rfreq} k^2 \norm{L_k^{-1} \hat f_k}_{H^2}^2
		\lesssim \sum_{k \in \Zodd \setminus \rfreq} \abs{\omega k}^{2 \ceil{3 - \tilde \gamma}} \norm{\hat f_k}_2^2
		= \norm{\partial_t^{\ceil{3 - \tilde \gamma}} f}_2^2 < \infty.
	\end{align*}
	Combined, we have shown that $\partial_t^l w = \partial_t^l \rproj w + \partial_t^l (1 - \rproj) w$ lies in $L^2 \cap L^\infty$.

	\textit{Part 4:}
	Using \eqref{eq:loc:1} we see
	\begin{align*}
		\partial_x^2 w = V(x) \partial_t^2 w + \partial_t^2 \calG \ast w + h \partial_t^2 u^3.
	\end{align*}
	where the right-hand side (and therefore $\partial_x^2 w$) is infinitely time-differentiable by the above and \cref{lem:polynomial_multiplier}.
\end{proof}

We are now ready to prove the main theorem.

\begin{proof}[Proof of \cref{thm:main}]
	Let $v \in L^{\frac43}(\R \times \T; \R)$ be a nonzero solution to \eqref{eq:dual_wave}, which exists by \cref{thm:existence}.
	Then $u \coloneqq h^{-\frac14} v^{\frac13}$ is the corresponding solution of \eqref{eq:variational_wave}. Define $w \coloneqq u$ for polarization \eqref{eq:polarization:1}, or define $w$ by \eqref{eq:wave_reconstruction:polarization2} for polarization \eqref{eq:polarization:2}. Then $\partial_t^l \partial_x^m w \in L^2 \cap L^\infty(\R \times \T; \R)$ for $l \in \N_0$ and $m \in \set{0,1,2}$ by \cref{lem:regularity:u} or \cref{lem:regularity:w}, and $w$ solves \eqref{eq:scalar_wave}.
	
	We define $\partial_t^{-1}$ on $\T$ as the Fourier multiplier with symbol $\frac{1}{\ii \omega k}$. Then, recalling the ansatz \eqref{eq:maxwell_ansatz} and Maxwell's equations \eqref{eq:maxwell},\eqref{eq:material:generic} we obtain the following formula for the corresponding solutions $\bfE, \bfD, \bfB, \bfH$ of Maxwell's equations:
	\begin{align*}
		\bfE(x, y, z, t) &= w(x, t - \frac{1}{c} z) \cdot \begin{pmatrix}
			0 \\ 1 \\ 0
		\end{pmatrix},
		\\ \bfB(x, y, z, t) &= - \partial_t^{-1} \nabla \times \bfE
		= - \begin{pmatrix}
			\frac{1}{c} w(x, t - \tfrac{1}{c} z)
			\\ 0 \\ W_x(x, t - \tfrac{1}{c} z)
		\end{pmatrix},
		\\ \bfH(x, y, z, t) &= \frac{1}{\mu_0} \bfB(x, y, z, t),
		\\ \bfD(x, y, z, t) &= \epsilon_0 (\bfE + \bfP(\bfE))
	\end{align*}
	where we have set $W \coloneqq \partial_t^{-1} w$. From the regularity of $w$ it follows that
	\begin{align*}
		\partial_t^l \bfE \in H^2 \cap W^{2,\infty}(\Omega; \R^3),
		\qquad
		\partial_t^l \bfB, \partial_t^l \bfH \in H^1 \cap W^{1,\infty}(\Omega; \R^3),
		\qquad
		\partial_t^l \bfD \in L^2 \cap L^\infty(\Omega; \R^3)
	\end{align*}
	for $l \in \N_0$ and domains of the form $\Omega = \R \times [y_1, y_2] \times [z_2, z_2] \times [t_1, t_2]$.
	Note that $\bfD$ may not be differentiable in space because the material coefficients $g_0, g_1, h$ of $\bfP$ were not assumed to be smooth.
\end{proof}

%% file: appendix.tex

\section{}\label{sec:appendix}

Here we give a proof of \cref{lem:basic_calc_properties} on fundamental properties of the transform $T$ given by \cref{thm:functional_calc}. We also prove \cref{thm:example:periodic,thm:example:localized}. 

\begin{proof}[Proof of Lemma~\ref{lem:basic_calc_properties}]
	\textit{(a):} First let $f \in H^2(\R; \C)$ be compactly supported. 
	Then 
	\begin{align*}
		T[L f](\lambda) 
		= \int_\R L f(x) \Psi(x; \lambda) \der[V]{x}
		= \int_\R f(x) L \Psi(x; \lambda) \der[V]{x}
		= \int_\R f(x) \lambda \Psi(x; \lambda) \der[V]{x}
		= \lambda T[f](\lambda).
	\end{align*}
	Reversely, let $T[f]$ be compactly supported. Then 
	\begin{align*}
		T^{-1}[\lambda T[f]]
		= \int_\R T_i[f] \lambda \Psi_j(x; \lambda) \der \mu_{ij}(\lambda) 
		= \int_\R T_i[f] L \Psi_j(x; \lambda) \der \mu_{ij}(\lambda) 
		= L f
	\end{align*}
	lies in $L^2_V(\R; \C)$. By approximation, we obtain the result for general $f$.
	
	\textit{(b):} If $f \in H^2(\R; \C)$, using (a) and that $T$ is an isometry we have
	\begin{align*}
		\int_\R \abs{f'}^2 \der x
		= \int_\R L f \cdot \overline{f} \der[V] x
		= \int_\R T^{-1}[\lambda T[f](\lambda)] \cdot \overline{f} \der[V] x
		= \int_\R \lambda T_i[f](\lambda) \overline{T_j[f](\lambda)} \der \mu_{ij}(\lambda).
	\end{align*}
	For general $f$, one can argue by approximation.
	
	\textit{(c):} 
	First let $\lambda_0 \in \C \setminus \supp(\mu)$. Then $\frac{1}{\lambda - \lambda_0}$ is a bounded symbol on $\supp(\mu)$, thus 
	\begin{align*}
		(L - \lambda_0)^{-1} = T^{-1}\Bigl[\frac{1}{\lambda - \lambda_0} T[\impvar](\lambda)\Bigr]
	\end{align*}
	is bounded, i.e., $\lambda_0 \in \rho(L)$ holds.
	
	Now let $\lambda_0 \in \supp(\mu)$, i.e., $\lambda_0 \in \supp(\mu_{ij})$ for some $i, j \in \set{1, 2}$. By definition we have
	\begin{align*}
		\abs{\mu_{ij}}(B_\eps(\lambda_0)) > 0
	\end{align*}
	for all $\eps > 0$, with $\abs{\mu_{ij}}$ denoting the total variation of $\mu_{ij}$.
	Split $B_\eps(\lambda_0) = E^+ \cup E^-$ with $E^+ \cap E^- = \emptyset$ into positive and negative part according to the measure $\mu_{ij}$. Then we have 
	\begin{align*}
		0 < \abs{\mu_{ij}}(B_\eps(\lambda_0))
		= \mu_{ij}(E^+) - \mu_{ij}(E^-)
		\leq \sqrt{\mu_{ii}(E^+) \mu_{jj}(E^+)}
		+ \sqrt{\mu_{ii}(E^-) \mu_{jj}(E^-)}
	\end{align*}
	since $\mu(E^{\pm})$ are positive semidefinite matrices. This implies 
	\begin{align*}
		\mu_{ii}(B_\eps(\lambda_0)) 
		= \mu_{ii}(E^+) + \mu_{ii}(E^-) > 0,
		\qquad
		\mu_{jj}(B_\eps(\lambda_0)) 
		= \mu_{jj}(E^+) + \mu_{jj}(E^-) > 0
	\end{align*}
	since $\mu_{ii}, \mu_{jj}$ are a nonnegative measures, that is, $\lambda_0 \in \supp(\mu_{ii}) \cap \supp(\mu_{jj})$. 
	
	We denote by $\delta_i$ the $i$-th unit vector, and consider the function $f \coloneqq T^{-1}[\bbone_{B_\eps(\lambda_0)} \delta_i]$. Then since
	\begin{align*}
		\frac{\norm{(L - \lambda_0) f}_{L^2_V}}{\norm{f}_{L^2_V}}
		= \frac{\norm{(\lambda - \lambda_0) \bbone_{B_\eps(\lambda_0)}}_{L^2(\mu_{ii})}}{\norm{\bbone_{B_\eps(\lambda_0)}}_{L^2(\mu_{ii})}} \leq \eps
	\end{align*}
	for arbitrary $\eps > 0$, $L - \lambda_0$ has no bounded inverse, i.e., $\lambda_0 \in \sigma(L)$ holds.
\end{proof}

\begin{proof}[Proof of \cref{thm:example:periodic,thm:example:localized}]
	For both theorems, we check that the assumptions of \cref{thm:main} are satisfied. 
	First \ref{ass:h}, \ref{ass:g0} hold by definition, and the same is true for \ref{ass:perturbed_nonlin} in \cref{thm:example:periodic}, and for \ref{ass:localized_nonlin} in \cref{thm:example:localized}.
	Next we calculate the Fourier coefficients
	\begin{align*}
		\hat \calN_k 
		= \int_\T T \dist(t, T \Z) \ee^{-\ii \omega k t} \der t
		= \int_0^T \dist(t, T \Z) \ee^{-\ii \omega k t} \der t
		= \begin{cases}
			\frac{T^2}{4}, & k = 0, 
			\\ 0, & k \neq 0 \text{ even},
			\\ - \frac{T^2}{2 k^2 \pi^2}, & k \text{ odd},
		\end{cases}
	\end{align*}
	and see that \ref{ass:calN} holds with $\alpha = 2$. The spectrum of the operator $L$ is investigated in \cite[Appendix~C]{henninger_ohrem_reichel}. There, it is shown that $L$ has a spectral gap about $\omega^2 k^2$ for each $k \in \Zodd$, that the size of the gap grows linearly in $\abs{k}$, and moreover that the point spectrum grows quadratically, i.e., \ref{ass:spec} holds with $\gamma = 1$, $\beta = \frac12$ and some $\delta > 0$. Next \ref{ass:L4_embed} holds since $\alpha + \gamma - 2 = 1 > \frac12 = \beta$.
	
	We calculate
	\begin{align*}
		\hat \calG_k(x)
		= g_1^\per(x) \int_0^T \cos(\omega t) \abs{\cos(\omega t)} \ee^{-\ii \omega k t} \der t
		= \begin{cases}
			0, & k \text{ even},
			\\ g_1^\per(x) \frac{4 T (-1)^n}{(4 k - k^3) \pi}, & k = 2 n + 1 \text{ odd},
		\end{cases}
	\end{align*}
	and find $\abs{\hat \calG_k(x)} \leq \frac{C}{k^2} \frac{1}{\omega^2 \abs{k}} V(x)$. Since in \ref{ass:g1} we require $d < \delta$, this only shows \ref{ass:g1} for sufficiently large $k$.
	Similarly, \ref{ass:add_for_polarization_2} holds with $s = 2$ but only for sufficiently large $k$. 

	This restriction to large frequencies is not an issue:
	Because $\rfreq = \set{k \in \Zodd \colon \hat \calN_k \neq 0} = \Zodd$ is infinite, by \cref{prop:multiplicity} and its proof there exist infinitely many distinct breathers that are supported only on large frequencies.
\end{proof}